\documentclass[preprint,11pt]{elsarticle}
\usepackage[margin=1in]{geometry} 
\usepackage{amsmath}
\usepackage{setspace}
\usepackage[utf8]{inputenc}
\usepackage{hyperref}
\usepackage[english]{babel}
\usepackage{placeins}
\usepackage[usenames, dvipsnames]{color}
\usepackage{algorithm}
\usepackage[noend]{algpseudocode}
\makeatletter
\def\BState{\State\hskip-\ALG@thistlm}
\makeatother
\usepackage{float}
\usepackage{rotating}
\usepackage{booktabs}
\usepackage{comment}
\usepackage{xcolor}
\usepackage[colorinlistoftodos]{todonotes}

\usepackage{tikz}
\usepackage{mathtools}
\usetikzlibrary{shapes}
\usetikzlibrary{plotmarks}
\usepackage{pgfplots}
\usepackage{pgfplotstable}
\usepackage{pgfplotstable}
\usepackage{tikz-network}
\usetikzlibrary{plotmarks}
\usepgfplotslibrary{dateplot, statistics}

\usepackage{csvsimple}
\usetikzlibrary{calc}
\usepackage{filecontents}
\usepackage{datatool,xfp}

\usepackage{enumitem}

\usepackage{xcolor}

\usepackage[most]{tcolorbox}
\newtcolorbox{highlighted}{colback=yellow,coltext=red,breakable}

\usepackage[title]{appendix}

\usepackage{nomencl}
\makenomenclature
\usepackage{footnote}
\usepackage[flushleft]{threeparttable}

\usepackage{graphicx}
\usepackage{amssymb}
\usepackage{cleveref}
\usepackage{hyperref}
\usepackage[utf8]{inputenc}
\usetikzlibrary{decorations.pathreplacing}

\usepackage{lineno}
\usepackage[utf8]{inputenc}
\PassOptionsToPackage{draft}{graphicx}
\usepackage[draft]{graphicx}
\usepackage{subcaption}
\usepackage{nomencl}
\usepackage{comment}
\usepackage{algpseudocode}
\usepackage{natbib}
\setcitestyle{authoryear,open={(},close={)}}
\usepackage{amsthm}
\newtheorem{prop}{Proposition}
\newtheorem{ass}{Assumption}

\newtheorem{theorem}{Theorem}
\algnewcommand\And{\textbf{and }}
\algnewcommand\Or{\textbf{or }}
\usepackage{cleveref}
\crefformat{section}{\S#2#1#3} 
\crefformat{subsection}{\S#2#1#3}
\crefformat{subsubsection}{\S#2#1#3}

\newcommand{\fg}[1]{\textcolor{ForestGreen}{#1}}

\usepackage{multirow}

\journal{Transportation Research Part A: Policy and Practice}

\begin{document}
	\onehalfspacing
	\begin{frontmatter}

		\title{Planning of integrated mobility-on-demand and urban transit networks}

		\author[pk]{Pramesh Kumar*}
		\author[pk]{Alireza Khani}
		\address[pk]{Department of Civil, Environmental and Geo-Engineering, University of Minnesota, Twin Cities, MN}
		
		\begin{abstract}
             We envision a multimodal transportation system where Mobility-on-Demand (MoD) service is used to serve the first mile and last mile of transit trips. For this purpose, the current research formulates an optimization model for designing an integrated MoD and urban transit system. The proposed model is a mixed-integer non-linear programming model that captures the strategic behavior of passengers in a multimodal network through a passenger assignment model. It determines which transit routes to operate, the frequency of the operating routes, the fleet size of vehicles required in each transportation analysis zone to serve the demand, and the passenger flow on both road and transit networks. A Benders decomposition approach with several enhancements is proposed to solve the given optimization program. Computational experiments are presented for the Sioux Falls multimodal network. \fg{The results show a significant improvement in the congestion in the city center with the introduction and optimization of an integrated transportation system. The proposed design allocates more vehicles to the outskirt zones in the network (to serve the first mile and last mile of transit trips) and more frequency to the transit routes in the city center. The integrated system significantly improves the share of transit passengers and their level of service in comparison to the base optimized transit system. The sensitivity analysis of the bus and vehicle fleet shows that increasing the number of buses has more impact on improving the level of service of passengers compared to increasing the number of MoD vehicles. Finally, we provide managerial insights for deploying such multimodal service.}
		\end{abstract}
		
				
		\cortext[mycorrespondingauthor]{Corresponding author}
		\fntext[myfootnotetel]{Email: kumar372@umn.edu}
		\fntext[myfootnotetel]{Tel: (612) 461-1643}
		\fntext[myfootnotetel]{Web: \href{http://umntransit.weebly.com/}{http://umntransit.weebly.com/}}

	\end{frontmatter}
	
	\newpage
	\section{Introduction}\label{sec:intro}
    The introduction of Mobility-on-Demand (MoD) services such as Uber, Lyft, and others as transportation alternatives has created many opportunities as well as challenges. On one hand, they provide a seamless mobility service with just a few taps on a cellphone application. On the other hand, \fg{they have} increased congestion in densely populated areas due to an increase in the relocation and pickup trips made by the participating drivers in the network  (\citet{Laris2019}). Furthermore, the transportation agencies envision the introduction of Autonomous Vehicles as a shared mobility service in the near future (\citet{Motavalli2020}), which would lead to severe congestion in densely populated areas as predicted by various simulation studies (\citet{Levin2015, Fagnant2016c, Levin2017}). \\
 
 	Public transportation, which can carry multiple passengers, is widely considered as a practical solution to the congestion problem by reducing vehicle-miles traveled (VMT) on roads (\citet{Aftabuzzaman2015}). However, due to its fixed routes and schedules, limited network coverage, and \fg{long} waiting time, sometimes, it is less attractive to travelers in comparison to the auto mode. The limited network coverage makes it difficult or sometimes impossible to access transit service in some areas. This inaccessibility problem is also known as the \textit{first mile/last mile (FMLM) problem for transit}. The problem is commonly faced by travelers commuting from low-density areas where transit service is not available or less frequent because of the economic in-viability of providing such service. \\

 	\fg{A few studies have argued that the MoD service provided using autonomous vehicles would become a competitor of public transit mode (\citet{Chen2016a, Levin2015, Mo2020a}), reducing its ridership, and other studies have even raised the question of whether urban mobility is possible without it (\citet{OECD2015, Mendes2017})}. On the contrary, \citet{8569381} showed that the integration of the MoD system with transit system could help in achieving better results, such as a significant reduction in travel time, emissions, and costs as compared to the standalone MoD system. Through the current research, we also envisage an integrated MoD and transit system that aims to achieve the following potential benefits:
    
    \begin{enumerate}
    	\item Providing fast and reliable mobility in low-density areas (i.e., by providing a first mile/last mile service) by means of characteristics of MoD service such as demand responsiveness, fleet repositioning, and reachability. 

	\item Allocation of resources from less congested areas to providing high-frequency transit service in congested areas through such integration. 

	\item Using existing transit infrastructure to reduce the number of vehicles needed for serving trips.

	\item Reducing congestion and carbon emissions in the network, improving the mobility of travelers, and reducing the overall cost of providing transit service. 
    	    	
    \end{enumerate}
       
    To achieve the above-mentioned benefits, we focus on the strategic planning of the transportation network that allows for intermodal trips with the first or last leg of the trips being served by the MoD service. To be specific, we try to answer questions such as which transit routes to operate when MoD vehicles are deployed to serve the FMLM connection, what should be the size of the vehicle fleet to be deployed, and what should be the frequency of operating transit routes. We attempt to answer these questions and make the following contributions through this article:
    
	\begin{enumerate}
	    \item Propose a passenger assignment model that predicts the travel behavior of passengers in a multimodal network. This step extends the idea of a hyperpath transit assignment model proposed by \cite{Spiess1989} to a multimodal transportation system with on-demand services. 
		\item Develop an optimization model to decide which transit routes to operate, frequency of operating transit routes, and MoD fleet size required to serve the FMLM of trips. 
		\item Develop a fast Benders decomposition implementation that uses efficient cutting planes to solve the large instances of the current problem.
		\item Conduct numerical experiments to show the efficacy of the proposed model and solution methods and discuss the steps to implement such service in practice. 
	\end{enumerate}

	\section{Related work}\label{sec:rw}
	The passenger journeys that consist of auto, as well as transit mode, create a new mode of transportation known as \textit{intermodal} or \textit{multimodal} transportation. \fg{In intermodal transportation, different agencies are responsible for providing the service for different legs of the passenger journey involving multiple modes. On the other hand, in multimodal transportation, only one agency is responsible for providing service for the entire passenger journey traversed using multiple modes.} The research on modeling such integrated transportation has been an active area of research for several decades (\citet{wilson1972statewide}). Many of these studies are focused on solving the transit FMLM problem by designing a multimodal transportation system. This includes designing a demand responsive transit feeder service (\citet{Wang2017a, maheo2017benders, Cayford2004, Koffman2004, Lee2017, Quadrifoglio2008, Shen2012, Li2009a}), using park-and-ride facilities (\citet{Khani2012, Khani2020, Kumar2021}), and integrating ridesharing and transit (\citet{Masoud2017, Stiglic2018, Bian2019, Ma2019, Chen2020, KUMAR2021102891}).\\
	
	Recently, the studies are being focused on modeling the integration of MoD and transit service for future mobility. They can be divided into two categories: simulation-based and optimization-based approaches. Under a set of assumptions on vehicle operations operations and dispatching strategies, the simulation-based studies simulate the passenger flow to assess the service quality of providing such mobility service (\citet{doi:10.1177/0361198120936267}). By using a four-step travel demand simulation model, \citet{Levin2015} predicted that the transit ridership will decrease and the number of personal vehicles will sharply increase as a result of the repositioning of vehicles resulting in congestion on the network. \citet{vakayil2017integrating} developed a simulation model that accounts for transit frequency, transfer costs, and MoD fleet re-balancing to use MoD as the FMLM solution to the transit mode. Their results show that such an integrated system can reduce VMT in the network by up to 50\%. \citet{Mendes2017} developed an event-based simulation model to compare the performance of the MoD system with the light rail system under the same demand patterns, alignment, and operating speed. They found that 150 vehicles with 12 passenger capacity would be needed to match the 39-vehicle light rail system if operated as a demand responsive system. Similar findings were also shown by the simulation model developed by \citet{Basu2018}. They showed that the introduction of MoD will act as the competitor of mass transit, however, to reduce congestion and maintain a sustainable urban transportation system, it cannot replace mass transit. \citet{Shen2018} also proposes and simulates an integrated autonomous vehicle and public transportation system based on the fixed modal split assumption. Using Singapore's organizational structure and demand characteristics, they propose to preserve high-demand bus routes while re-purposing low-demand bus routes and using shared MoD as an alternative. They found that the integrated system has the potential of serving the trips with less congestion, less passenger discomfort, and economically viable service. \citet{Wen2018a} included mode choice and various vehicle capacities and hailing strategies in an agent-based model to provide insights into fleet sizing and frequency of transit routes for the integrated system. A few studies have used an optimization-based approach to developing an integrated passenger flow model. \citet{8569381} developed a network flow model for intermodal service  that couples the interaction between MoD and transit by maximizing social welfare. Using this model, they proposed a tolling scheme for this intermodal system that helps in reducing the travel time, costs, and emissions as compared to standalone vehicle mode. \citet{LIU2019648} used Bayesian optimization to predict the mode choice of passengers in such a multimodal transportation system. \\

	The above-cited studies show that an integrated MoD and transit system can provide an efficient mode of transportation that is sustainable, fast, eco-friendly, and economically viable. The design of such a system requires solving a \textit{multimodal transportation network design problem} that can decide various aspects of MoD and transit modes. The problem of designing transit routes and their corresponding frequencies, which is commonly referred to as the Transit Network Design Problem (TNDP) or Line Planning Problem (LPP) in the literature, is itself a complex problem (\citet{CEDER1986331, baaj1991ai}). There has been a significant amount of research in modeling TNDP and developing solution algorithms for it. For a review on transit network design literature, we refer the interested reader to \citet{Guihaire2008}. Some aspects of the multimodal network design problem have been explored in a related research problem known as \textit{hub and arc location problem} (\citet{Maheo2019,Campbell2005, Campbell2005a}). For example, \citet{Maheo2019} proposed the design of a hub and shuttle public transit system in Canberra. They formulated a mixed-integer program to design high-frequency bus routes between key-hubs, where the first mile or last mile of trips is covered by the shuttles. However, the hub and arc location problem have a major limitation of not able to capture passenger behavior in the transit network. Recently, a couple of studies have proposed models for the transit network design in the context of integrated MoD and transit system (\citet{manser2017public, Pinto2020, doi:10.1287/trsc.2020.0987}). \citet{Pinto2020} develops a bi-level optimization model to design a transit network integrated with MoD service. The upper-level optimization problem modifies the frequency of the transit routes and determines the fleet size of MoD service and the lower-level model simulates the passenger trajectories based on a simulation-based traveler assignment model. Due to the complexity of the model, they presented a heuristic approach to solving the current problem. \citet{doi:10.1287/trsc.2020.0987} presents various aspects of this problem and develops a path-based mixed-integer programming model to decide which sections of the transit routes to operate and locate the transfer stops to allow for intermodal trips in the network. Due to an enormous number of possible paths in the network, they solve the current model using a branch-and-price approach.\\
	
	The design of an integrated MoD and transit system is an important problem that can influence the future mobility of travelers. \fg{Recent studies have made important contributions to this complex problem but have several limitations. It includes designing networks by routing passengers on the shortest path using a multi-commodity network flow model and presenting solution algorithms not applicable to solving large-scale instances of the problem. We attempt to address these limitations and outline the motivation to pursue the current research in the following points:}
	
	\begin{enumerate}
		\item Before designing the integrated system, we should understand how passengers would behave in an integrated system. It is common for studies to use the classic multi-commodity flow model to predict the behavior of travelers in the network design \fg{(\citet{conforti2014integer}, Chapter~4)}. This may be true if passenger trajectories are completely influenced by the mobility provider. However, this is certainly not applicable in the case of transit systems when passengers try to reduce the expected travel time based on wait time, travel time, and fare. Through this study, we extend the idea of hyperpath passenger assignment for a multimodal transportation system.
		
		\item We develop a mixed-integer optimization model that incorporates the multimodal passenger assignment and evaluates various aspects of an integrated system. The optimization program is difficult to solve, and we need efficient techniques to solve this problem. For this purpose, an exact method based on the Benders Decomposition is proposed to solve the large-scale instances of the problem. The method improves the classic Benders decomposition strategy by precluding the infeasibility cuts and including new cuts, such as disaggregated cuts, multiple cuts, and clique/cover cuts.  
	\end{enumerate}

	The rest of the article is structured as follows. \cref{sec:prelim} discusses the notations and definitions used in this article. Then, we present the multimodal passenger assignment model, which is incorporated in the design model of the integrated MoD and transit system in \cref{sec:design}. The solution algorithm to solve the design model is discussed in \cref{sec:sm}, which is followed by the results of numerical experiments conducted on Sioux Falls network \fg{in \cref{sec:6}}. Finally, conclusions and recommendations for future research are presented in \cref{sec:cf}.

	\section{Preliminaries and Background} \label{sec:prelim}	
	In this section, we get familiarize ourselves with the notations and concepts to be used in this article. Let us begin by considering a multimodal transportation network characterized by a digraph $G(N, A)$, where $N$ denotes the set of nodes that includes road intersections $N_R$, transit stops/stations $N_T$, and centroids of traffic analysis zones $Z$\footnote{A \textit{traffic analysis zone} (TAZ) or simply a zone is a geographical area where the demand is assumed to be concentrated on its centroid.} and $A$ denotes the set of links. We associate every node $i \in N$ in the network with exactly one zone $Z(i)$. The set of links coming out and going into a node $i \in N$ are denote by $FS(i) = \{(i, j): (i, j) \in A\}$ and $BS(i) = \{(j, i): (j, i) \in A\}$ respectively. Let $\mathfrak{d}: N \times N \mapsto \mathfrak{R}_{+}$ be the distance function between two nodes in the network. Depending on the mode, the links are also divided into three categories, namely transit, road, and walking links represented by $A_T, A_R, \text{and } A_W$ respectively. Let $O \subset Z$ and $D \subset Z$ be the subsets of centroids representing the origins and destinations respectively. The demand between various origin-destination pairs is represented by $\{d_{od}\}_{(o, d) \in O \times D}$. The overall network can be divided into three sub-networks which are described below:
	
	\begin{enumerate}
		\item \textit{Transit network}: The transit network is characterized by the subgraph $G_T(N_T, A_T)$ which consists of a set of candidate transit lines/routes denoted by the set $L$.  The terms ``route" and ``line" are used interchangeably throughout this article. Each line $l \in L$ is composed of a set of stops $N^l_T \subset N_T$ which are connected by edges $A^l_T \subset A_T$. The network also consists of transfer  links $A^{tr}_T$ between two nodes if the walking distance between those is less than the acceptable walking distance $\zeta$ (say 0.75mi), i.e., $A^{tr}_T = \{(n_1, n_2) \in N \times N : n_{1} \in N^{l_{1}}_T, n_{2} \in N^{l_{2}}_T \text{ for some } l_1, l_2 \in L \text{ s.t. } l_1 \ne l_2 \text{ and } \mathfrak{d}(n_1, n_2) \le \zeta  \}$.		
		\item \textit{Road network}: The road network is characterized by the subgraph $G_R(N_R, A_R)$, where $N_R$ denotes the set of nodes and $A_R$ denotes the set of links in the road network. 
		\item \textit{Walking links}: The walking links consists of access, egress, and mode transfer links. The \textit{access} and \textit{egress links} are defined as $A^a = \{(n_1, n_2) \in Z \times (N_T \cup N_R) : \mathfrak{d}(n_1, n_2) \le \zeta\}$ and $ A^e = \{(n_1, n_2) \in (N_T \cup N_R) \times Z  : \mathfrak{d}(n_1, n_2) \le \zeta\}$ respectively. Similarly, the \textit{mode transfer links} are defined as $A^m =\{(n_1, n_2) \in N_R \times N_T: \mathfrak{d}(n_1, n_2) \le \zeta\} \cup \{(n_1, n_2) \in N_T \times N_R: \mathfrak{d}(n_1, n_2) \le \zeta\}$. The access and egress walking links connect the centroids of various zones with the road/transit nodes and vice-versa, whereas mode transfer links are used to transfer between nodes of various modes. 
	\end{enumerate}

	\subsection{Costs}
	There is a subset of nodes in the network where passengers have to wait for the service. The collection of head nodes of links in the sets $A^a$, $A^{tr}_T$, and $A^m$ constitutes the \textit{waiting nodes} $N^w$.	Let us assume that $c: A \mapsto \mathfrak{R}_{+}$ and $w: N^w \mapsto \mathfrak{R}_{+}$ denote the cost (e.g., walking time, in-vehicle time, and fare) associated with the links in $A$ and wait time associated with the nodes in $N^w$ respectively. The cost of links is known beforehand (and is computed by adding the travel time and possible fare multiplied by the value of time). On the other hand, the wait time depends on the availability of MoD or transit service. 
	
	\subsection{Wait time computation}
	Unlike a personal vehicle, the MoD or transit service is not readily available, and passengers have to wait to access these services. So, it is important to quantify the expected wait time of these services, the computation of which is discussed below:
	
	\subsubsection{MoD service}\label{sec:waitMoD}
	We assume MoD operations in a network as a queuing system to compute the average wait time experienced by the passengers to access such service. The average wait time may not be justified for the planning of day-to-days operations but can be used to approximate the actual wait time experienced by the passengers for long-term strategic planning of the network, which is the focus of the current study. Therefore, we consider a stationary state of an MoD system, where the number of waiting customers $\mathcal{C}$ and vacant vehicles $\mathcal{V}$ are time-invariant. Using the Cobb-Douglas production function, the matching time between the customers and the vacant vehicles can be expressed as a function of $\mathcal{C}$ and $\mathcal{V}$.
	
	\begin{equation}
	m^{c-v} = \mathcal{A} (\mathcal{V})^{\alpha_1} (\mathcal{C})^{\alpha_2}
	\end{equation}
	
	\noindent where, $\alpha_1$ and $\alpha_2$ are defined as the elasticities of the matching function and $\mathcal{A}$ is a parameter specific to a zone, which is a function of the market area divided by the running speed in that zone (\citet{Zha2016}). According to Little's law, the long-term average number of customers/drivers in a stationary system is equal to the long-term average arrival rate $Q$ multiplied by the average wait time ($w^c/w^t$) that a customer/driver spends in the system before being matched (\citet{Zha2016}). 
	
	\begin{eqnarray}
	\mathcal{V} = Q w^t \label{eq:1} \\
	\mathcal{C} = Q w^c \label{eq:2}
	\end{eqnarray}
	
	Using \eqref{eq:2} and assuming $\alpha_1 = \alpha_2 \approx 1$ (\citet{douglas1972}), we can represent the stationary state ($m^{c-v} = Q$) as below:
	\begin{eqnarray}
	Q = \mathcal{A} \mathcal{V} (Q w^c)\\
	\implies   w^c = \frac{1}{\mathcal{A} \mathcal{V}} \label{eq:3}
	\end{eqnarray}
	Equation \eqref{eq:3} shows that the average wait time of customers wait in a zone to access the MoD service is a function of the vacant number of vehicles. To achieve the desired level of service (i.e., average wait time), a transportation agency needs to provide $\mathcal{V}$ vehicles at any point in time.

	\subsubsection{Transit service}\label{sec:waitTransit}
	Let us now discuss the wait time computation to access transit service at the head node of an access or transfer link in the transit network. Let $f: A_T \mapsto \mathfrak{R}$ be the frequency of the transit line associated with various links of the transit network. Let $\mathfrak{g}_i(w)$ be the probability distribution function of the wait time for line $i$.  According to \citet{larson1981urban}, for the passengers arriving randomly at a node, the probability density function of the wait time of line $i$ is related to the headway or bus inter-arrival time distribution $\mathfrak{h}_i(h)$ as:
	
	\begin{equation}\label{eq:4}
	\mathfrak{g}_i(w) = \frac{\int_{w}^\infty \mathfrak{h}_i(h) dh}{\mathbb{E}[\mathfrak{h}_i] }
	\end{equation}
	
	\noindent To evaluate the wait time distribution, we make the following assumptions:
	
	\begin{ass}\label{as:1}
		The inter-arrival time of a transit line $i \in L$ follows an exponential distribution with rate $f_i$.
	\end{ass}

	\begin{ass}\label{as:2}
		Passengers want to minimize the expected wait time to get to their destination. Therefore, at any node, passengers waiting to be served by the transit service have selected a list of attractive transit lines that can help them to get to their destination. 
	\end{ass}

	Both assumption \ref{as:1} and \ref{as:2} are common in the transit assignment literature (e.g., see \citet{Desaulniers2007}). By using the assumption \ref{as:1} and equation \eqref{eq:4}, one can evaluate the distribution function of the wait time $\mathfrak{g_i}(w)$ as: 
	
	\begin{equation}\label{eq:6}
		\mathfrak{g_i}(w) = f_ie^{-f_iw}, w \ge 0
	\end{equation}

	\begin{prop}\label{prop:1} (\citet{Spiess1989, gentile2005route}) Assuming that a passenger waiting at node $n \in N^w$ is served by the set of attractive transit lines $FS^{*}(n)$ and let $\mathfrak{F} = \sum_{j\in FS^{*}(n)} f_{j}$. With assumptions \ref{as:1} and \ref{as:2}, the following holds:
		\begin{enumerate}
			\item The probability that a passenger would choose transit line $i \in FS^{*}(n)$  is given by
			\begin{equation}\label{eq:7}
			P_i =  \frac{f_i}{\mathfrak{F}}
			\end{equation}
			\item The expected wait time conditional to boarding line $i \in FS^{*}(n)$ is given by	
			\begin{equation}
			EW_i = \frac{f_i}{\mathfrak{F}^2}
			\end{equation}
			\item The probability of wait time at node $n$ follows an exponential distribution with rate $\mathfrak{F}$. Therefore, the expected wait time at stop $n$ is given by $EW_n = \frac{1}{\mathfrak{F}}$.
		\end{enumerate}
	\end{prop}
	
	\begin{proof}
		See Appendix \ref{app:1}.
	\end{proof}

\subsubsection{Combined MoD and transit service wait time}\label{sec:waitCombined}
	Before discussing the computation of the expected wait time involving both modes, we need to make an assumption about the wait time distribution of MoD service by utilizing the value of the average wait time of MoD service calculated in equation \eqref{eq:3}.
	\begin{ass}
		The wait time distribution of MoD service for passengers waiting at node $n$ follows an exponential distribution with rate $f_{MoD} = \mathcal{A}_{Z(n)} \mathcal{V}_{Z(n)}$, where $Z(n)$ is the zone associated to node $n$ and $\mathcal{V}_{Z(n)}$ is the number of vehicles deployed in zone $Z(n)$. 
	\end{ass}
	A passenger waiting at the head node of an access link faces the choice between MoD or transit mode. This is because the wait time of both services can vary based on the frequency provided, and a passenger will include one or both modes in their strategy to reduce the overall expected cost. This assumption simplifies the operation of MoD service as a transit service available at any stop of the network. The following proposition evaluates the expected wait time of that passenger.
	
	\begin{prop}\label{prop:2}
		Given that the wait time for transit and MoD mode follow an exponential distribution with rate $\mathfrak{F}$ and $f_{MoD}$ respectively and $\mathbb{F} = \mathfrak{F} + f_{MoD}$ , the following holds:
		\begin{enumerate}
			\item The probabilities of taking transit and MoD are given by $P_{MoD} = \frac{f_{MoD}}{\mathbb{F}}$ and $P_{transit} = \frac{\mathfrak{F}}{\mathbb{F}}$ respectively. 
			\item The expected wait time of the passenger departing from an access node $n$ served by both MoD and transit service is given by 
			$EW_n = \frac{1}{\mathbb{F}}$
		\end{enumerate}
	\end{prop}
	
	\begin{proof}
		See Appendix \ref{app:1}.
	\end{proof}

	\begin{figure}[h!]		
		\begin{subfigure}[t]{0.90\textwidth}            
			\centering
			\begin{tikzpicture}[thick, scale=0.8, every node/.style={scale=0.6}]
			\draw[solid, ->]  (-11.2,-2)  -- (-10,-2);	
			\node[text width=4cm] at (-8.2,-2) {Road Link};
			\draw[dash dot, ->] (-7.8,-2) -- (-6.8,-2);	
			\node[text width=4cm] at (-5,-2) {Access link};
			\draw[solid, ->, blue] (-4.8,-2) -- (-4,-2);	
			\draw[solid, ->, green] (-3.8,-2) -- (-3,-2);
			\draw[solid, ->, red] (-2.8,-2) -- (-2,-2);	
			\node[text width=4cm] at (-0.1,-2) {Transit link};
			\end{tikzpicture}
			\begin{tikzpicture}[thick, scale=0.6, every node/.style={scale=0.5}]   
			\node[shape=circle,draw=black] (11) at (-2,0) {$Z_1$}; 
			\node[shape=circle,draw=black] (1) at (0,0) {$1$};
			\node[shape=circle,draw=black] (6) at (6,2) {$2$};
			\node[shape=circle,draw=black] (4) at (6, -2) {$4$};
			\node[shape=circle,draw=black] (8) at (12,0) {$8$};
			
			\node[shape=circle,draw=black] (9) at (12, -2) {$9$};
			\node[shape=circle,draw=black] (12) at (16,0) {$Z_2$}; 
			\node[shape=circle,draw=black] (13) at (2,4) {$R_1$}; 
			\node[shape=circle,draw=black] (14) at (6,4) {$R_2$};
			\node[shape=circle,draw=black] (15) at (14,4) {$R_3$};  
			
			\node[shape=circle,draw=black] (16) at (2,-4) {$R_4$}; 
			\node[shape=circle,draw=black] (17) at (6,-4) {$R_5$};
			\node[shape=circle,draw=black] (18) at (14,-4) {$R_6$}; 
			
			\path [->, dash dot] (11) edge node[rotate=55, yshift=7pt, xshift =-1pt] {} (16)  ;
			\path [->, dash dot] (15) edge node[rotate=55, yshift=7pt, xshift =-1pt] {} (12) ;
			\path [->, dash dot] (18) edge node[rotate=55, yshift=7pt, xshift =-1pt] {} (12) ;
			
			\path [->, dash dot] (14) edge node[rotate=55, yshift=7pt, xshift =-1pt] {} (6)  ;
			\path [->, dash dot] (17) edge node[rotate=55, yshift=7pt, xshift =-1pt] {} (4) ;

			\path [->, dash dot] (11) edge node[rotate=55, yshift=7pt, xshift =-1pt] {} (13)  ;
			\path [->] (13) edge node[rotate=55, yshift=7pt, xshift =-1pt] {} (14)  ;
			\path [->] (14) edge node[rotate=55, yshift=7pt, xshift =-1pt] {} (15) ;
			\path [->] (16) edge node[rotate=55, yshift=7pt, xshift =-1pt] {} (17)  ;
			\path [->] (17) edge node[rotate=55, yshift=7pt, xshift =-1pt] {} (18)  ;

			\path [->] (1) edge node[rotate=55, yshift=7pt, xshift =-1pt] {} (6) [color=red];
			\path [->] (1) edge node[rotate=55, yshift=7pt, xshift =-1pt] {} (4) [color=green];
			\path [->,bend right] (6) edge node[rotate=55, yshift=7pt, xshift =-1pt] {} (8) [color=blue];
			\path [->, bend left] (6) edge node[rotate=55, yshift=7pt, xshift =-1pt] {} (8)  [color=red];
			\path [->, dash dot] (8) edge node[rotate=55, yshift=7pt, xshift =-1pt] {} (12) ;
			\path [->, dash dot] (9) edge node[rotate=55, yshift=7pt, xshift =-1pt] {} (12) ;

			\path [->] (4) edge node[rotate=55, yshift=7pt, xshift =-1pt] {} (9) [color=green];

			\path [->, dash dot] (11) edge node[rotate=55, yshift=7pt, xshift =-1pt] {} (1);
			\end{tikzpicture}
			
			\caption{Multimodal network}
			\label{fig:stNet}
		\end{subfigure}  

		\begin{subfigure}[t]{0.90\textwidth}    
			\centering
				\begin{tabular}{l|l}
				\hline
				\textbf{Line}  & \textbf{Frequency} \\ \hline
				Red   & 1/6       \\ 
				Green & 1/2       \\ 
				Blue  & 1/3       \\ \hline
			\end{tabular}     
			\hspace{1mm}
			\begin{tabular}{l|l}
				\hline
				\textbf{Zone} & \textbf{Vehicles} \\ \hline
				 $Z_1$    & 100       \\ 
				 $Z_2$    & 50       \\ \hline
			\end{tabular}			
			\caption{Frequency and number of vehicles}
			\label{tab:line}
		\end{subfigure}
		\caption{An illustrative example of a multimodal network}
		\label{fig:stab}
	\end{figure}

	To get more insights into the wait time computation, let us consider an example. Figure \ref{fig:stab}(a) shows an illustration of a multimodal transportation network. It consists of 2 zones, 6 nodes and 8 links as part of the road network, and 5 nodes and 10 links as part of the transit network. The transit network has 3 transit lines (color-coded) whose frequencies are shown in Figure \ref{fig:stab}(b). There are 100 and 50 vehicles deployed in zone 1 and 2 respectively. By using Prop \ref{prop:1}, we can evaluate the probability of passengers taking various transit lines in the network.  For example, the probabilities of choosing red line and green line at stop 1 are $\frac{1/6}{1/6 + 1/2} = 0.25$ and $\frac{1/2}{1/6 + 1/2} = 0.75$ respectively. The expected wait time at stop 1 is equal to $12/8 = 1.5$ minutes. Similarly, using Proposition \ref{prop:2}, the probabilities of choosing MoD and transit at $Z_1$ are $\frac{0.0017*100}{0.0017*100 + 8/12} = 0.2$  and $0.8$ respectively (assuming $\mathcal{A}_1 = 0.0017$). The overall expected wait time at node $Z_1$ is 1.19 minutes which is less than 1.5 minutes by only considering transit service as part of the strategy. \\
	
	We further use Proposition \ref{prop:1} and \ref{prop:2} to formulate the multimodal passenger assignment model. For this purpose, we extend the frequency-based transit assignment model proposed by \citet{Spiess1989} to a multimodal transportation system. Before moving forward, we must make the following assumptions:
 	
 	\begin{ass}\label{as:4}
 		\begin{enumerate}[label=(\alph*)]
 			\item Ridepooling is not allowed, i.e., the MoD service serves one passenger at a time.
 			\item The transit lines are assumed to have unlimited capacity. 
 			\item Passengers want to reduce their expected generalized travel cost consisting of travel time, wait time, and fare to get to their destination.
 		\end{enumerate} 		
 	\end{ass}
 
 	The ridepooling problem requires matching of customers using a specific algorithm. This is an important aspect to accurately estimate the cost of day-to-day operations. Nevertheless, ignoring ridepooling will give us an upper bound on the number of vehicles required to serve various zones. The modeling of passenger behavior while incorporating the capacity constraints (congestion) is a difficult problem. The congestion is important to consider since it causes denied boarding, which leads to increased wait time, travel time, and discomfort. Several authors have tried to include congestion into frequency-based transit assignment models through various approaches, namely, discomfort function (\cite{Spiess1989}), effective frequency (\cite{de1993transit, Cominetti2001, cepeda2006frequency, Leurent2014}), and failure-to-board probabilities (\cite{Kurauchi2003}). Despite the effort, there is no tractable closed-form of congested frequency-based transit assignment model. On the other hand, it would not be ideal to include transit vehicle capacity constraints into the assignment program (e.g., in \citet{Szeto2014b}) because doing so may lead to unrealistic passenger behavior, which previous studies on congested frequency-based transit assignments were trying to avoid. Therefore, we use an uncapaciated assignment for the design problem. Assumption \ref{as:4}(c) is a common in the assignment literature. The relaxation of above assumptions are research topics in their own right, therefore, a discussion on possible ways to relax them is provided in \cref{sec:cf}. To proceed further, let us define a variable $\{g_{ik}\}_{i \in N, k \in D}$ as below:
 	
 	\[ g_{ik} = \begin{cases} 
 	d_{ik},  & \text{if } i \ne k, (i, k) \in O \times D   \\	
 	-\sum_{o \in O} d_{ok},  & \text{if } i = k  \\		
 	0,  &  \text{otherwise} \\
 	\end{cases}
 	\]
 	
	Furthermore, let us denote $v_{ak}$ and $W_{ik}$ as the flow of passengers on link $a \in A$ and waiting at node $i \in N^w$ resp. destined to $k \in D$. The assignment optimization program is presented below:
	\begin{subequations}\label{eq:ma}
		\begin{align}
			& \underset{v, W}{\text{minimize}}
			& & \sum_{k \in D} \left(\sum_{a \in A} c_{a} v_{ak} + \sum_{i \in N^w} W_{ik} \right) \\
			& \text{subject to}
			& & \sum_{a \in FS(i)}  v_{ak} = \sum_{a \in BS(i)}  v_{ak} +  g_{ik}, \forall i \in N, \forall k \in D  \label{eq:cons}\\
			& & & v_{ak} \le  f_a W_{ik}, \forall a \in FS(i): a \in A_T, \forall i \in N^w, \forall k \in D  \label{eq:prop1} \\
			& & & v_{ak} \le  \mathcal{A}_{Z(i)} \mathcal{V}_{Z(i)} W_{ik}, \forall a \in FS(i): a \in A_R, \forall i \in N^w, \forall k \in D  \label{eq:prop2}\\
			& & &  v_{ak} \ge 0, \forall a \in A, \forall k \in D \label{eq:32}
		\end{align}
	\end{subequations}
	The assignment program \eqref{eq:ma} minimizes the total expected link costs and wait time at waiting nodes experienced by the passengers in a mulimodal network subject to the flow conservation constraint at each node \eqref{eq:cons}, flow proportion constraints \eqref{eq:prop1}-\eqref{eq:prop2}, and the non-negativity and binary constraints \eqref{eq:32}. The flow proportion constraints uses the probability of selecting an option $a \in FS(i)$ (if that option is a part of the strategy of the passengers traveling to destination $k \in D$) and multiplies it with the number of passengers waiting at that node. Note that the probability of selecting an option (MoD or transit line) is calculated in Proposition \ref{prop:2}. 

	\section{Design of an integrated MoD and transit system }\label{sec:design}
	In this section, we present an optimization model incorporating the assignment program proposed in previous section for the design of an integrated MoD and transit system. The optimization program is formalized as a Mixed Integer Non-linear Program (MINLP). In this model, we determine which transit routes to keep operating among the current transit routes in the city network, decide the optimal frequency of those operating routes, and finally, determine the fleet size of vehicles required to provide MoD service in various zones. Note that one can also include new candidate transit routes as part of the design plan. The sets, parameters, and decision variables for the optimization model are summarized in Table \ref{tab:dec}. \fg{The 0.01 vehicles in the set of possible number of vehicles to be deployed in a zone is a dummy element to represent that no vehicles are assigned in a zone, and that zone can be served by transit service only. This will help design an efficient solution algorithm for the design problem (see Proposition~\ref{prop:5}).} \\

	\begin{table}
		\caption{Sets, decision variables and parameters used in the design model}
		\label{tab:dec}				
		\begin{tabular}{r c p{10cm} }
			\toprule
			\multicolumn{3}{c}{\underline{Sets}}\\
			\multicolumn{3}{c}{}\\
			$\mathfrak{B}$ & $\triangleq$ & Set of binary values \\
			$L$ & $\triangleq$ & Set of candidate transit lines \\
			$\Theta = \{2, 3, 4, 6, 12\}$ & $\triangleq$ & Set of possible frequencies of a line (buses/hr) \\
			$\Omega = \{0.01, 50, 100, 200, 500\}$ & $\triangleq$ & Set of possible number of vehicles deployed in a zone \\
			
			\multicolumn{3}{c}{}\\
			\multicolumn{3}{c}{\underline{Parameters}}\\
			\multicolumn{3}{c}{}\\
			$\bar{B}$ & $\triangleq$ & Total number of buses available \\
			$\bar{F}$ & $\triangleq$ & Total number of vehicles available \\

			\multicolumn{3}{c}{}\\
			\multicolumn{3}{c}{\underline{Decision Variables}}\\
			\multicolumn{3}{c}{}\\
			
			$x_l$ & $=$ & \(\left\{\begin{array}{rl}
			1,  & \text{if line $l \in L$ is kept operating} \\
			0,  & \text{otherwise} \end{array} \right.\)\\

			$y_{lf}$ & $=$ & \(\left\{\begin{array}{rl}
			1,  & \text{if frequency $f \in \Theta$ is adopted for line $l \in L$} \\
			0,  & \text{otherwise} \end{array} \right.\)\\

			$\mathcal{N}_{zn} $ & $=$ & \(\left\{\begin{array}{rl}
			1,  & \text{ if a fleet of size $n \in \Omega$ is deployed in zone $z \in Z$} \\
			0,  & \text{otherwise} \end{array} \right.\)\\
			
			$v_{ak}$ & $=$ & Flow of passengers on link $a \in A$ destined to $k \in D$ \\
			
			$W_{ik}$ & $=$ & Wait time of passengers waiting at node $i \in N^w$ destined to $k \in D$ \\

			\bottomrule
		\end{tabular}
	\end{table}

	The design of an integrated transit and MoD system should consider both passenger and operator perspectives. The operator's perspective is to provide the service at minimum cost, and the passengers' perspective is to minimize the overall cost of travel (including travel time, wait time, and fare). Based on these perspectives, the design optimization model is presented as \eqref{eq:11}. The objective function is the sum of the total expected travel cost and wait time experienced by the passengers in the network. The mapping $\mathcal{B}: L \times \Theta \mapsto \mathbb{N}$ used in \eqref{eq:busesAv} is defined as $\mathcal{B}(l, f) = \left(f \times \sum_{a \in A_T^l} 2 t_a\right)$, which describes that the  number of buses required to provide frequency $f \in \Theta$ for a line $l \in L$ is equal to the product of the frequency and round trip travel time. \eqref{eq:busesAv} constrain the total number of buses needed to be less than or equal to $\bar{B}$, which can be evaluated for a given budget. \eqref{eq:12} describes the flow conservation constraints at every node for every destination. For a given MoD and bus fleet assignment, \eqref{eq:13}-\eqref{eq:14} describe the passenger flow on each link based on the frequency of the bus route and MoD service. A frequency value can be assigned to a route if that route is kept operating as constrained by \eqref{eq:15}. \eqref{eq:21} describe that exactly one of the fleet sizes can be adopted for each zone. \eqref{eq:16} constrain the required number of vehicles to be less than or equal to $\bar{F}$. Finally, \eqref{eq:17}, \eqref{eq:22} and \eqref{eq:18}-\eqref{eq:20} are the non-negativity constraints of the flow, wait time being free variables, and binary constraints of design variables respectively. One can also incorporate other constraints related to the budget of operating MoD and transit service but for the sake of simplicity, we do not include them here.

	\begin{subequations}\label{eq:11}
		\begin{align}
		& \underset{v, W, x, y, \mathcal{N}}{\text{minimize}}
		& & \sum_{k \in D} \left(\sum_{a \in A} c_{a} v_{ak} + \sum_{i \in N^w} W_{ik}\right) \\
		& \text{subject to}
		& & \sum_{l \in L} \sum_{f \in \Theta} \mathcal{B}(l, f) \times y_{lf}  \le \bar{B}   \label{eq:busesAv}\\
		& & & \sum_{a \in FS(i)}  v_{ak} = \sum_{a \in BS(i)}  v_{ak} +  g_{ik}, \forall i \in N, \forall k \in D  \label{eq:12}\\
		& & & v_{ak} \le  \left(\sum_{f\in \Theta} f y_{l(a)f}\right) W_{ik}, \forall a \in FS(i): a \in A_T, \forall i \in N^w, \forall k \in D \label{eq:13}\\
		& & & v_{ak} \le  \mathcal{A}_{Z(i)} \left(\sum_{n \in \Omega}  n\mathcal{N}_{Z(i)n}\right) W_{ik}, \forall a \in FS(i): a \in A_R, \forall i \in N^w, \forall k \in D \label{eq:14}\\
		& & & \sum_{f \in \Theta} y_{lf} = x_l, \forall l \in L \label{eq:15}\\
		& & & \sum_{n \in \Omega}\mathcal{N}_{zn} = 1, \forall z \in Z \label{eq:21}\\
		& & & \sum_{z \in Z}\left(\sum_{n \in \Omega}  n\mathcal{N}_{zn}\right) \le \bar{F} \label{eq:16}\\
		& & &  v_{ak} \ge 0, \forall a \in A, \forall k \in D \label{eq:17}\\
		& & & W_{ik} \text{ free }, \forall i \in N^w, \forall k \in D \label{eq:22}\\
		& & &  x_l \in \mathfrak{B}, \forall l \in L \label{eq:18}\\
		& & &  y_{lf} \in \mathfrak{B}, \forall f \in \Theta,  \forall l \in L \label{eq:19}\\
		& & & \mathcal{N}_{zn} \in \mathfrak{B}, \forall n \in \Omega, z \in Z \label{eq:20}
		\end{align}
	\end{subequations}

 	The optimization program \eqref{eq:11} is a mixed-integer non-linear program (MINLP). The non-linearity arise from the constraints \eqref{eq:13}-\eqref{eq:14}. It is computationally difficult to solve this program for large instances, which can be attributed to the integer constraints \eqref{eq:18}-\eqref{eq:20} and the bilinear constraints \eqref{eq:13}-\eqref{eq:14}. The bilinear constraints are particularly difficult to handle due to the non-convex nature even if the integrality constraints of the involved variables are relaxed. Fortunately, in this case, the non-convexity arises due to the product of continuous and binary variables, which can be exactly relaxed by employing McCormick relaxations. Let $t_{faik} = y_{l(a)f} W_{ik}, \forall f \in \Theta, \forall a \in FS(i): a \in A_T, \forall i \in N^w, \forall k \in D$ and $\omega_{ink} = \mathcal{N}_{Z(i)n} W_{ik}, \forall n \in \Omega,  \forall i \in N^w \cap N_R, \forall k \in D$. Further, let us assume that there exists a finite upper and lower bound on the variable $W_{ik}$ , i.e.,  $\underline{W}_{ik} \le W_{ik} \le \overline{W}_{ik}$. Then, $t_{faik}$ and $\omega_{ink}$ can be expressed as the set of linear constraints \eqref{eq:22a}-\eqref{eq:23} and \eqref{eq:24}-\eqref{eq:25} respectively:
 		\begin{subequations}
 		\begin{align}
 			\overline{W}_{ik} - W_{ik}  + t_{faik} - \overline{W}_{ik}y_{l(a)f} \ge 0 \label{eq:22a}\\
 			\overline{W}_{ik} y_{l(a)f} - t_{faik} \ge 0 \label{eq:26}\\
 			t_{faik} - \underline{W}_{ik}y_{l(a)f} \ge 0 \label{eq:27}\\
 			W_{ik} - \underline{W}_{ik} - t_{faik} + \underline{W}_{ik}y_{l(a)f} \ge 0  \label{eq:23}
 		\end{align}			
 	\end{subequations}
 	
 	\begin{subequations}
 		\begin{align}
 			\overline{W}_{ik} - W_{ik}  + \omega_{ink} - \overline{W}_{ik} \mathcal{N}_{Z(i)n} \ge 0 \label{eq:24}\\
 			\overline{W}_{ik} \mathcal{N}_{Z(i)n} - \omega_{ink} \ge 0\\
 			\omega_{ink} - \underline{W}_{ik}\mathcal{N}_{Z(i)n} \ge 0\\
 			W_{ik} - \underline{W}_{ik} - \omega_{ink} + \underline{W}_{ik}\mathcal{N}_{Z(i)n} \ge 0 \label{eq:25}
 		\end{align}			
 	\end{subequations}
 
	\section{Solution methodology} \label{sec:sm}
	After reformulating.0 the bilinear constraints \eqref{eq:13}-\eqref{eq:14}, the resulting model is a Mixed Integer Linear Program (MILP). The program is still difficult to solve efficiently for large instances. However, the structure of the problem allows us to use decomposition techniques such as Benders decomposition to efficiently solve it. In this section, we present the details of the Benders reformulation for this problem, along with the proposed algorithmic enhancements. 
	
	\subsection{Benders Reformulation}
	Benders decomposition (\citet{Geoffrion1972}) is an elegant way of solving a large scale MILP by iteratively solving two simpler subproblems: the relaxed master problem (RMP),  which is a relaxation of the original problem and a subproblem (SP) which provides inequalities/cuts to strengthen the RMP. The subproblem should possess strong duality properties. Let us consider the network design problem described in the previous section. For a given feasible value of design decision variables $\hat{x}, \hat{y}, \hat{\mathcal{N}}$ and with $\underline{W}_{ik} = 0$ (wait time cannot be negative), we can rewrite the original problem as a \textit{Benders subproblem} \eqref{eq:28}.

		\begin{subequations}\label{eq:28}
		\begin{align}
		&  z^{SP} (\hat{x}, \hat{y}, \hat{\mathcal{N}}) =\underset{v, W, \omega, t}{\text{min}}
		& &  \sum_{k \in D} \left(\sum_{a \in A} c_{a} v_{ak} + \sum_{i \in N^w} W_{ik}\right) \\
		& \text{subject to}
		& & \sum_{a \in FS(i)}  v_{ak} = \sum_{a \in BS(i)}  v_{ak} +  g_{ik}, \forall i \in N, \forall k \in D \label{eq:29}\\
		& & & v_{ak} \le  \left(\sum_{f\in \Theta} f t_{faik} \right), \forall a \in FS(i): a \in A_T, \forall i \in N^w, \forall k \in D \label{eq:33}  \\
		& & &  W_{ik}  - t_{faik} \le  \overline{W}_{ik} (1-\hat{y}_{l(a)f}),  \forall f \in \Theta, \forall a \in FS(i): a \in A_T, \forall i \in N^w, \forall k \in D \label{eq:40} \\
		& & &  t_{faik} \le \overline{W}_{ik} \hat{y}_{l(a)f},  \forall f \in \Theta, \forall a \in FS(i): a \in A_T, \forall i \in N^w, \forall k \in D\label{eq:41}  \\
		& & & W_{ik}  - t_{faik}  \ge 0,  \forall f \in \Theta, \forall a \in FS(i): a \in A_T, \forall i \in N^w, \forall k \in D  \label{eq:34}  \\	
		& & & v_{ak} \le  \mathcal{A}_{Z(i)} \left(\sum_{n \in \Omega}  n \omega_{ink}\right), \forall a \in FS(i): a \in A_R, \forall i \in N^w, \forall k \in D \label{eq:36} \\
		& & & W_{ik} - \omega_{ink}   \le  \overline{W}_{ik}(1-\hat{\mathcal{N}}_{Z(i)n}), \forall n \in \Omega, \forall i \in N^w \cap N_R, \forall k \in D\\
		& & & \omega_{ink} \le \overline{W}_{ik} \hat{\mathcal{N}}_{Z(i)n}, \forall n \in \Omega, \forall i \in N^w \cap N_R, \forall k \in D\\
		& & & W_{ik} - \omega_{ink} \ge 0, \forall n \in \Omega, \forall i \in N^w \cap N_R, \forall k \in D \label{eq:37}\\
		& & & v_{ak} \ge 0, \forall a \in A, \forall k \in D\\
		& & & t_{faik}\ge 0, \forall f \in \Theta,  \forall a \in FS(i): a \in A_T, \forall i \in N^w, \forall k \in D \label{eq:40a}\\
		& & & \omega_{ink} \ge 0, \forall n \in \Omega, \forall i \in N^w \cap N_R, \forall k \in D\label{eq:30}
		\end{align}
	\end{subequations}
	
	Let $\mathcal{X}^{SP} = \{(v, W, \omega, t): \eqref{eq:29}-\eqref{eq:30} \}$ be the feasible region of the Benders subproblem \eqref{eq:28}.  Further, let us denote $\mathcal{X}^{MA} = \{(v, W):\eqref{eq:cons}-\eqref{eq:32}\}$ as the feasible region of the multimodal assignment linear program. We can show the following result:
	
	\begin{prop}\label{prop:4}
		The projection of the feasible region of the subproblem \eqref{eq:28} on to the space of $v$ and $W$ is same as the feasible region of the multimodal assignment problem \eqref{eq:ma} i.e., 
		$$proj_{v, W} \mathcal{X}^{SP} = \mathcal{X}^{MA} $$
	\end{prop}

	\begin{proof}
		See Appendix \ref{app:1}.	
	\end{proof}
	Proposition \ref{prop:4} shows that one can use the efficient \citet{Spiess1989}'s primal-dual algorithm designed for the transit assignment problem to solve the current Benders subproblem.  To speed up the process of Benders decomposition by avoiding feasibility cuts, we need to put some restrictions so that the subproblem \eqref{eq:28} is always feasible. 
	
	\begin{prop}\label{prop:5}
		Given that $0 \notin \Omega$ and  $G_R(N_R, A_R)$ is connected, then $\mathcal{X}^{SP}$ is non-empty for any given feasible value of $\hat{x}, \hat{y}, \hat{\mathcal{N}}$. 
	\end{prop}
	
	\begin{proof}
			See Appendix \ref{app:1}.
	\end{proof}

	Proposition \ref{prop:5} makes the Benders subproblem feasible for any feasible value of $\hat{x}, \hat{y}, \hat{\mathcal{N}}$. This is an important result to make the Benders decomposition implementation faster.\\

	Let $\{\mu_{ik}\}$, $\{\lambda_{aik}^1\}$, $\{\lambda_{faik}^2\}$, $\{\lambda_{faik}^3\}$, $\{\lambda_{faik}^4\}$, $\{\lambda_{aik}^5\}$, $\{\lambda_{nik}^6\}$, $\{\lambda_{nik}^7\}$, and $\{\lambda_{nik}^8\}$ be the dual variables associated with the constraints \eqref{eq:29} -\eqref{eq:37} respectively. Then, the dual of the subproblem DSP can be stated as below:
	
	\begin{subequations}\label{eq:42}	
		\begin{align}
		&  z^{DSP} (\hat{x}, \hat{y}, \mathcal{\hat{N}}) =\underset{\mu, \lambda}{\text{max}}
		& & \sum_{k \in D} \left[\sum_{i \in N} \mu_{ik} g_{ik} + \sum_{i \in N^w} \sum_{a \in FS(i): a \in A_T} \sum_{f\in \Theta} \left(\overline{W}_{ik} (1-\hat{y}_{l(a)f})\lambda^2_{faik} \right.  \right. \nonumber\\
		& & & \left. \left. +   \overline{W}_{ik} \hat{y}_{l(a)f} \lambda^3_{faik} \right)\right) + \sum_{i \in N^w \cap N_R} \sum_{n \in \Omega} \left(\overline{W}_{ik}(1-\hat{\mathcal{N}}_{Z(i)n})\lambda_{nik}^6 \right. \nonumber\\
		& & & \left. \left. + \overline{W}_{ik} \hat{\mathcal{N}}_{Z(i)n} \lambda^7_{nik} \right) \right] \\
		& \text{subject to}
		& & \mu_{ik} - \mu_{jk} +  \lambda^1_{aik} + \lambda^5_{aik} \le c_{a}, \forall a = (i, j) \in A,  \forall k \in D \label{eq:43}  \\
		& & & -\sum_{f \in \theta} \sum_{\underset{a \in A_T} {a \in FS(i):}} \left(\lambda^2_{faik} + \lambda^4_{faik} \right) + \sum_{n \in \Omega} \left( \lambda^6_{nik} + \lambda^8_{nik} \right)  = 1, {\footnotesize   \forall , \forall i \in N^w , \forall k \in D }\\
		& & &  -f\lambda^1_{aik} - \lambda^2_{faik} + \lambda^3_{faik} - \lambda^4_{faik} \le 0, \forall f \in \Theta, \forall a \in FS(i): a \in A_T, \forall i \in N^w, \forall k \in D  \\
		& & &  -n\lambda^5_{aik} - \lambda^6_{nik} + \lambda^7_{nik} - \lambda^8_{nik} \le 0, \forall n \in \Omega, \forall i \in N^w \cap N_R, \forall k \in D  \\
		& & & \lambda^1_{aik}, \lambda^5_{aik} \le 0, \forall a \in FS(i), \forall i \in N^w, \forall k \in D\\
		& & & \lambda^2_{faik} , \lambda^3_{faik} \le 0, \lambda^4_{faik} \ge 0, \forall f \in \Theta, \forall a \in FS(i): a \in A_T, \forall i \in N^w, \forall k \in D\\
		& & & \lambda^6_{nik}, \lambda^7_{nik} \le 0,  \lambda^8_{nik} \ge 0, \forall n \in \Omega, \forall i \in N^w \cap N_R, \forall k \in D \label{eq:44}
		\end{align}
	\end{subequations}
	
	Let us denote the feasible region of DSP as $\Pi  = \{(\mathbf{\mu}, \mathbf{\lambda}^1, \mathbf{\lambda}^2, \mathbf{\lambda}^3, \mathbf{\lambda}^4, \mathbf{\lambda}^5, \mathbf{\lambda}^6, \mathbf{\lambda}^7, \mathbf{\lambda}^8): \eqref{eq:43}-\eqref{eq:44} \}$. Note that $\Pi$ does not depend on the value of $x, y, \mathcal{N}$. From Proposition \ref{prop:5}, we know that SP is always feasible for any given feasible value of $(\hat{x}, \hat{y}, \hat{\mathcal{N}})$, then by linear programming duality, DSP should be bounded. The implication is that the polyhedron describing $\Pi$ is bounded and can be described as the convex hull of a set of extreme points only (from Minkowski-Weyl's theorem on characterization of polyhedra (\citet[Chapter~3]{conforti2014integer})). Let $\{(\mathbf{\mu}^\pi,  (\mathbf{\lambda}^1)^\pi, (\mathbf{\lambda}^2)^\pi, (\mathbf{\lambda}^3)^\pi,(\mathbf{\lambda}^4)^\pi, (\mathbf{\lambda}^5)^\pi, (\mathbf{\lambda}^6)^\pi, (\mathbf{\lambda}^7)^\pi, (\mathbf{\lambda}^8)^\pi)\}_{\pi \in \mathcal{K}}$ be the set of extreme points of polytope  $\Pi$, where $\mathcal{K}$ represents the set of indices of extreme points. By applying an outer linearization procedure to the inner (sub) problem of the original problem, we can restate it as \eqref{eq:50}, which is referred to as the \textit{Benders Master problem} (MP). 
	
	\begin{theorem} (\citet{Benders1962})
		The problem \eqref{eq:11} can be reformulated as below:
			\begin{subequations}\label{eq:50}
			\begin{align}
			& \underset{x, y,\mathcal{N}, \eta}{\text{minimize}}
			& & \eta \\
			& \text{subject to}
			& & \sum_{l \in L} \sum_{f \in \Theta} \mathcal{B}(l, f) \times y_{lf}  \le \bar{B}  \label{eq:50a} \\
			& & & \sum_{f \in \Theta} y_{lf} = x_l, \forall l \in L \\
			& & & \sum_{n \in \Omega}\mathcal{N}_{zn} = 1, \forall z \in Z \\
			& & & \sum_{z \in Z}\left(\sum_{n \in \Omega}  n\mathcal{N}_{zn}\right) \le \bar{F} \label{eq:50b}\\
			& & & \eta \ge  \sum_{k \in D} \left[\sum_{i \in N} (\mu_{ik})^\pi g_{ik} + \sum_{i \in N^w} \sum_{a \in FS(i): a \in A_T} \sum_{f\in \Theta} \left(\overline{W}_{ik} (1-\hat{y}_{l(a)f})(\lambda^2_{faik})^\pi \right.  \right. \nonumber \\
			& & & \left. \left. +   \overline{W}_{ik} \hat{y}_{l(a)f} (\lambda^3_{faik})^\pi \right)\right) + \sum_{i \in N^w \cap N_R} \sum_{n \in \Omega} \left(\overline{W}_{ik}(1-\hat{\mathcal{N}}_{Z(i)n})(\lambda_{nik}^6)^\pi \right. \nonumber\\
			& & & \left. \left. + \overline{W}_{ik} \hat{\mathcal{N}}_{Z(i)n} (\lambda^7_{nik})^\pi \right) \right], \forall \pi \in \mathcal{K} \label{eq:50e}\\
			& & &  x_l \in \mathfrak{B}, \forall l \in L \label{eq:50c}\\
			& & &  y_{lf} \in \mathfrak{B}, \forall f \in \Theta,  \forall l \in L \\
			& & & \mathcal{N}_{in} \in \mathfrak{B}, \forall n \in \Omega, i \in Z \label{eq:50d}
			\end{align}
		\end{subequations}
	\end{theorem}

	\begin{proof}
		See \citet{Benders1962}.
	\end{proof}

	\subsection{Classic Benders decomposition implementation}\label{sec:bdAlg}
	The issue with the Benders reformulation is that there could be a large number of extreme points of the polyhedron associated with the feasible region of DSP, therefore, one applies an iterative process of solving two problems, namely, the relaxed master problem (RMP) and the subproblem (SP) repeatedly. The relaxed master problem is the master problem with constraints \eqref{eq:50e} being defined only for a subset of extreme points, i.e., $\mathcal{K}^{'} \subset \mathcal{K} $. The overall implementation of the classic Benders Decomposition is summarized in Algorithm \ref{alg:cb}. We start by finding the feasible value of $(x^0,y^0, \mathcal{N}^0)$. This can be done by solving \eqref{eq:50} without \eqref{eq:50e} and including a constraint $\eta \ge 0$. Then, in each iteration $t$, the algorithm solves RMP with the given set of extreme points and then SP with the current value of $(x^t, y^t, \mathcal{N}^t)$. Since RMP is relaxation and SP is solved for a feasible value $(x^t, y^t, \mathcal{N}^t)$, they provide a lower bound and upper bound respectively to the original problem. The subproblem also provides inequalities (optimality cuts) to strengthen the formulation of RMP in each iteration. Thus, it is guaranteed to have non-decreasing lower bounds. In our case, there are no feasibility cuts since our subproblem is always feasible (Proposition \ref{prop:5}). The algorithm terminates when both the upper bound and lower bound are close to each other.

	\begin{algorithm}[H]
		\caption{Classic Benders decomposition implementation}\label{alg:cb}
		\begin{algorithmic}[1]
			\State (\textit{Initialize}) Let $t = 0, UB = -\infty, LB = \infty, \mathcal{K}^{'} = \phi$. Assume an initial feasible value $(x^0, y^0, \mathcal{N}^0)$. Solve the SP \eqref{eq:28}, obtain the optimal dual solution and append that to set $\mathcal{K}^{'}$. 
			\While{$UB - LB > \epsilon$} \Comment{$\epsilon$ is the tolerance parameter}
			\State Set $t= t +1$. Solve RMP \eqref{eq:50} and obtain its optimal solution $(x^t, y^t, \mathcal{N}^t)$. 
			\State Set LB = $\eta$
			\State Solve SP \eqref{eq:28} for $(x^t, y^t, \mathcal{N}^t)$, obtain dual solutions and append that to $\mathcal{K}^{'}$. 
			\State Set $UB = \sum_{k \in D} \left(\sum_{a \in A} c_{a} v_{ak}^t + \sum_{i \in N^w} W_{ik}^t \right)$
			\EndWhile
		\end{algorithmic}
	\end{algorithm}

	\subsection{Enhanced Benders decomposition implementation}\label{sec:bdAlgEn}
	The classic Benders decomposition may take prohibitive computational effort to converge, thus making it difficult to solve the problem for large instances. The slow convergence can be attributed to the low strength of the optimality cuts, degeneracy in the subproblem, no guarantee of non-decreasing upper bounds in each iteration, or not formulating the problem "properly" (\citet{Saharidis2010a,Tang2013, Magnati1984}). To accelerate the Benders decomposition algorithm, we make use of several enhancements that are described below:

	\subsubsection{Use of multiple cuts via disaggregated cuts} \label{sec:dis}
	For this design problem, we can further utilize the decomposable structure of the Benders subproblem \eqref{eq:11} as it is decomposable for each destination $k \in D$. That is, we can solve several (smaller) subproblems and generate multiple optimality cuts for the master problem. The disaggregated cuts have a higher probability of finding facet-defining inequalities characterizing $\Pi$. For this purpose, we modify RMP to allow for the disaggregated cuts as \eqref{eq:51}:

		\begin{subequations}\label{eq:51}
			\begin{align}
			& \underset{x, y,\mathcal{N}, \eta}{\text{minimize}}
			& & \sum_{k \in D}\eta_k \\
			& \text{subject to}
			& & \eqref{eq:50a} -\eqref{eq:50b} \\
			& & & \eta_k \ge  \left[\sum_{i \in N} (\mu_{ik})^\pi g_{ik} + \sum_{i \in N^w} \sum_{a \in FS(i): a \in A_T} \sum_{f\in \Theta} \left(\overline{W}_{ik} (1-\hat{y}_{l(a)f})(\lambda^2_{faik})^\pi \right.  \right. \nonumber \\
			& & & \left. \left. +   \overline{W}_{ik} \hat{y}_{l(a)f} (\lambda^3_{faik})^\pi \right)\right) + \sum_{i \in N^w \cap N_R} \sum_{n \in \Omega} \left(\overline{W}_{ik}(1-\hat{\mathcal{N}}_{Z(i)n})(\lambda_{nik}^6)^\pi \right. \nonumber\\
			& & & \left. \left. + \overline{W}_{ik} \hat{\mathcal{N}}_{Z(i)n} (\lambda^7_{nik})^\pi \right) \right], \forall \pi \in \mathcal{K}_k, \forall k \in D \label{eq:51d} \\
			& & &  \eqref{eq:50c} -\eqref{eq:50d} 
			\end{align}
		\end{subequations}
	
	Note that by adding the disaggregated cuts for every destination, we can get back the optimality cuts defined in the classic Benders relaxed master problem. 
	
	\subsubsection{Use of multiple cuts via multiple solutions} \label{sec:mult}
	To further improve the convergence of the algorithm, \citet{BeheshtiAsl2019} used a strategy known as multiple cuts via multiple solutions. When solving the RMP, any commercial solver such as AIMMS or GUROBI can be asked to generate multiple solutions of an integer program (optimal as well as suboptimal) by using \texttt{pool solution option}. These multiple solutions can be used to generate multiple classic \eqref{eq:50e} or disaggregated cuts \eqref{eq:51d} to be added in next iteration of RMP. This strategy is expected to decrease the overall iterations and possibly the solution time of the algorithm.

	\subsubsection{Use of clique/cover cuts} \label{sec:cc}
	Due to the limited availability of bus and vehicle fleet, one can use the clique/cover cuts to tighten the feasible region of the master problem.
	
	\begin{prop}\label{prop:6}
		For every $n \in \Omega$, if $\lfloor \frac{\bar{F}}{n} \rfloor < \vert Z \vert$ then the clique inequality $\sum_{z \in Z} \mathcal{N}_{zn} \le \lfloor \frac{\bar{F}}{n} \rfloor$ is valid for \eqref{eq:11}. 
	\end{prop}
	
	\begin{proof}
			See Appendix \ref{app:1}. 
	\end{proof}

	If for any $n \in \Omega$, we have $\lfloor \frac{\bar{F}}{n} \rfloor > \vert Z \vert$, then the inequality $\sum_{z \in Z} \mathcal{N}_{zn} \le \lfloor \frac{\bar{F}}{n} \rfloor$ will be redundant and therefore, we do not add it to the model. \\

	The inequality which constrain the number of buses \eqref{eq:busesAv} is a Knapsack constraint. A set $C \subseteq L \times \Theta$ is a \textit{cover} for inequality \eqref{eq:busesAv} if $\sum_{(l, f) \in C} \mathcal{B}(l, f) > \bar{B} $ and it is \textit{minimal cover} if $\sum_{(l, f) \in C \backslash \{(l^{'}, f^{'})\}} \mathcal{B}(l, f) \le \bar{B}, \text{ for all } (l^{'}, f^{'}) \in C$

	\begin{prop}\label{prop:7}
		For any minimal cover $C\subseteq L \times \Theta$, the inequality $\sum_{(l, f) \in C} y_{lf} \le \vert C \vert -1$ is valid for \eqref{eq:11}.
	\end{prop}
	
	\begin{proof}
			See Appendix \ref{app:1}. 
	\end{proof}

	To generate some of the minimal cover cuts, one can use the heuristic given in Algorithm \ref{alg:cc}. In this algorithm, for each frequency $f \in \Theta$, we keep the list of lines $G$ for which the number of buses required to provide the frequency $f$ does not exceed $\bar{B}$. Then, any line which is not in $G$, along with $G$ forms a minimal cover. 
	
	\begin{algorithm}[H]
		\caption{Cover cut generation heuristic}\label{alg:cc}
		\begin{algorithmic}[1]
		\Procedure{}{}
			\State Compute the value of mapping $\mathcal{B}(l, f)$ for all $(l, f) \in L \times \Theta$.
			\State  $CC \gets \phi$
			\For{$f \in \Theta$}
			\State $G \gets []; temp \gets 0$
					\For{$l \in L: \mathcal{B}(l, f)$ in an ascending order}
					\State $temp = temp+ \mathcal{B}(l, f)$
						\If{$temp \le \bar{B}$}
							\State append $(l, f)$ to $G$
						\Else
							\State \textbf{break}
						\EndIf	
									
					\EndFor
					\For{$l \in L \backslash G$} 
					\State $C \gets G \cup \{(l, f)\}$
					\State append $C$ to $CC$
					\EndFor
			\EndFor
		\EndProcedure
		\Return{$CC$}
	\end{algorithmic}
	\end{algorithm}

	Furthermore, one can use other efficient techniques to produce maximal clique or minimal cover cuts for the problem.

	\subsubsection{Other recommendations}
	One of the problems with the Benders subproblem \eqref{eq:28} is that it assumes the value of $\overline{W}_{ik}$ as a given upper bound. The value of $\overline{W}_{ik}$ is a big-M introduced to relax the non-linearity in the original model. If the value of the big-M is not chosen properly, then one can face serious issues with the convergence of the algorithm. For example, choosing $\overline{W}_{ik} < W_{ik}$ can make the subproblem infeasible, and choosing $\overline{W}_{ik}$ too high would generate weak optimality cuts, which would increase the computational time of the algorithm. One way to avoid this issue is to solve the assignment problem \eqref{eq:ma} for given design variables $(x, y, \mathcal{N})$ and compute the optimal value of $W_{ik}$ and use that as an upper bound.  Further improvements in the Benders decomposition method can involve the use of \textit{pareto-optimal} cuts proposed by \citet{Magnati1984}. They help in avoiding the generation of multiple optimality cuts for a degenerate subproblem. We tried this strategy, however, we did not find any significant improvement in the solution time using these cuts, therefore, we do not discuss it here. Finally, when RMP is loaded with a large number of cuts we recommend removing the non-active cuts from the model by checking the dual value. There is no guarantee that they will not be generated again, but it will be faster to solve the RMP. The overall steps of the Benders implementation with possible acceleration techniques are summarized in Algorithm \ref{alg:acc}. 
	
	\begin{algorithm}[H]
		\caption{Enhanced Benders decomposition implementation}\label{alg:acc}
		\begin{algorithmic}[1]
			\State (\textit{Initialize}) Let $t = 0, UB = -\infty, LB = \infty, \mathcal{K}^{'}_k = \phi, \forall k \in D$.
			\State Prepare the master problem with clique and cover inequalities.		
			\State Assume an initial feasible value $(x^0, y^0, \mathcal{N}^0)$. Solve the SP \eqref{eq:28}, obtain the optimal dual solutions and append that to the set $\mathcal{K}^{'}_k, \forall k \in D$. 
			\While{$UB - LB > \epsilon$} \Comment{$\epsilon$ is the tolerance parameter}
			\State Set $t= t +1$. Solve RMP \eqref{eq:50}, obtain its optimal solution $(x^t_0, y^t_0, \mathcal{N}^t_0)$ and other optimal/suboptimal solutions $\{(x^t_s, y^t_s, \mathcal{N}^t_s)\}_{1 \le s \le l}$, where $l$ is specified by the user. 
			\State Set LB = $\sum_{k \in D}\eta_k$
			 	\For{$s = 0, 1, ..., l$}
			 		\State Solve SP \eqref{eq:28} for $(x^t_s, y^t_s, \mathcal{N}^t_s)$, obtain dual solution and append that to $\mathcal{K}^{'}_k, \forall k \in D$.
			 	\EndFor
		 	\State Set $UB = \sum_{k \in D} \left(\sum_{a \in A} c_{a} v_{ak}^t + \sum_{i \in N^w} W_{ik}^t \right)$
			\EndWhile
		\end{algorithmic}
	\end{algorithm}

	\section{Computational results}\label{sec:6}
	In this section, we present the computational study based on the model \eqref{eq:11}, \eqref{eq:50}, and acceleration techniques presented in \cref{sec:bdAlgEn}. We start by describing the details of the experiment used to show the application of the proposed method. Then, we present the details of the network design results in \cref{sec:netd}, which is followed by the comparison of the computational performance of the solver, the classic Benders implementation, and the enhanced Benders techniques described in \cref{sec:comp}. Finally, we discuss the results of the sensitivity analysis on two important parameters in the model, namely, the available fleet of buses $\bar{B}$ and vehicles $\bar{F}$ in \cref{sec:sens}, which is followed by the comparison of the performance of optimized existing transit system and proposed integrated system in \cref{sec:comp2}.

	\subsection{Experiment details}\label{sec:ins}
	The computational experiments are based on the Sioux Falls road and transit network. The road network has 24 nodes, whereas the static transit network has 384 stops. A walking distance of 0.5 miles is used to create walking links. An illustration of the two networks is shown in Figure \ref{fig:sfn} and the number of different types of links in the network is given in Table \ref{tab:net}. There are 12 candidate transit routes in the transit network. We consider the set of possible frequencies as $\Theta = \{2, 3, 4, 6, 12\}$  buses/hr to be assigned to any candidate transit route and possible vehicle fleet size to be assigned to any zone as $\Omega = \{0.01, 50, 100, 200, 500 \}$. A time-based fare of \$0.21/ min and a base fare of \$0.8 is assumed for the MoD service, whereas the transit fare is assumed to be a fixed value of \$2. To convert the monetary costs into time units, the value of travel time equal to 23\$/hr is used. The value of parameter $\mathcal{A}$ used in the wait time computations of MoD service is assumed to be equal to 0.0017 for all the zones (\citet{Yin2019y}). The available number of buses $\bar{B}$ and vehicles $\bar{F}$ are assumed to be 70 and 3,000 respectively. There are 576 O-D pairs in the network with a total number of trips equal to 36,060. All implementations are coded in Python 3.8 using Gurobi 9.0.1 as the optimization solver. The tests were executed on Intel(R) i7-7700 CPU running at 3.6 GHz with 32 GB RAM under a Windows operating system. \\

	\begin{figure}[h!]
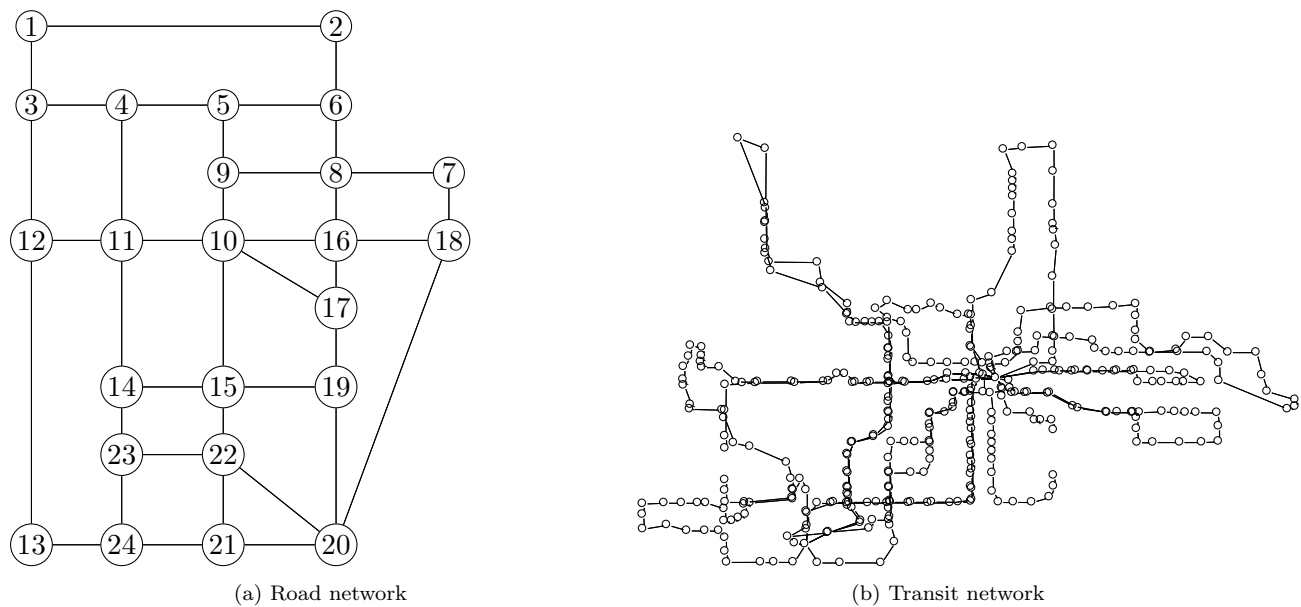
		
		\begin{subfigure}[t]{0.50\textwidth}            
    
			\caption{Road network}
		\end{subfigure}        
		\hfill
		\begin{subfigure}[t]{0.50\textwidth}            
			%
       
			\caption{Transit network}
		\end{subfigure}
		\caption{Sioux Falls network}
		\label{fig:sfn}
	\end{figure}

	 \begin{table}[]
 	 	\centering
 	 	\caption{Number of different types of links in the Sioux Falls multimodal network}
 	 	\label{tab:net}
 	 	\begin{tabular}{lc}
 	 		\hline
 	 		\textbf{Link type} & \textbf{Number of links} \\ \hline
 	 		Access ($A^a$)           & 243                      \\ 
 	 		Egress ($A^e$)          & 243                      \\ 
 	 		Road   ($A_R$)           & 76                       \\ 
 	 		Transit ($A_T$)         & 398                      \\ 
 	 		Transit transfer ($A^{tr}_T$)   & 368                      \\ 
 	 		Mode transfer \textbf{($A^m$)}     & 152                       \\ \hline
 	 	\end{tabular}
 	 \end{table}

	\subsection{Network design results}\label{sec:netd}
	We solve the network design problem \eqref{eq:11} for the instance explained in \cref{sec:ins}. The selected transit routes with their optimal frequency are given in Table \ref{tab:routeLoc}. Out of 12 candidate routes, 6 routes are kept operating. The transit network with active and inactive routes is shown in Figure \ref{fig:transitnet}. We observe that most of the routes are located in the central region of the network. All the routes have been assigned the highest frequency, i.e., 12 buses/hr, except route 8, which has been assigned a frequency of 3 buses/hr. To provide this service, 69 buses are required. The average number of vehicles deployed in each zone is given in Table \ref{tab:carAll}. In the optimal allocation of vehicles, it is decided not to deploy any vehicles in 10 zones out of 24 zones. Most of the zones have been allocated 200 vehicles providing an average wait time of 3 minutes. We further observe that the vehicles are deployed in the outskirt zones of the network where transit routes are not located.\\

	\begin{figure}[h!]
		\centering
		\includegraphics[width=1\linewidth]{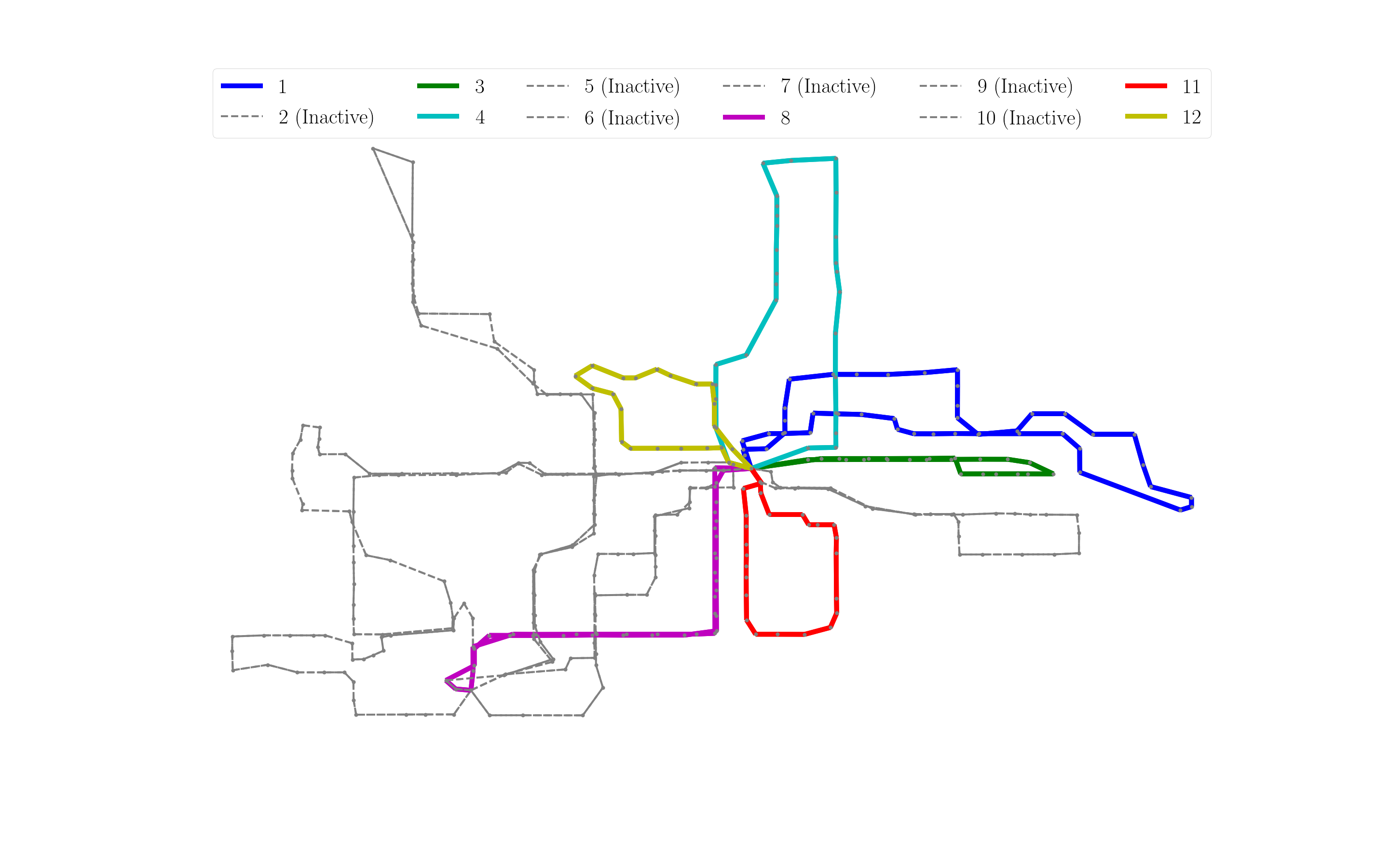}
		\caption{Transit routes (inactive routes are shown by dashed gray color)}
		\label{fig:transitnet}
	\end{figure}	
	
	\begin{table}[h!]
		\caption{Selected transit routes with their optimal frequency}
		\label{tab:routeLoc}
		\centering
		\begin{tabular}{ccccc}
			\hline
			\textbf{Route} & \textbf{Located?}& \textbf{Optimal} & \textbf{Average wait}\\
			& & \textbf{frequency} (buses/hr) & \textbf{time} (min) \\ \hline
			1     &   Yes    &   12         & 5                  \\ 
			2     &  No     &    -        & -                  \\ 
			3     & Yes      & 12            & 5                 \\ 
			4     & Yes      &   12              & 5            \\ 
			5     & No      &     -         & -                \\ 
			6     & No       &  -          &   -             \\ 
			7     & No     &  -            & -               \\ 
			8     & Yes       &   3          & 20               \\ 
			9     & No      &  -           & -                \\ 
			10    & No      &   -          &  -               \\ 
			11    & Yes      &   12          & 5               \\ 
			12    & Yes      &    12          & 5               \\ \hline
		\end{tabular}
	\end{table}

	\begin{table}[h!]
		\caption{vehicle allocation to different zones}
		\label{tab:carAll}
		\centering
		\begin{tabular}{ccc||ccc}
			\hline
			\textbf{Zone} & \textbf{Vehicles} & \textbf{Avg. Wait time (min)} & \textbf{Zone} & \textbf{Vehicles} & \textbf{Avg. Wait time (min)} \\ \hline
			1    & 200        & 3              & 13   & 200        & 3              \\ 
			2    & 100       & 6              & 14   & 200        & 3              \\ 
			3    & 200        & 3             & 15   & 200       & 3             \\ 
			4    & 200        & 3              & 16   & -       & -            \\ 
			5    & -        & -              & 17   & -       & -            \\ 
			6    & -        &  -             & 18   & 200        & 3              \\   
			7    & 200       & 3             & 19   & 200        & 3              \\ 
			8    & -        &  -             & 20   & 200        & 3              \\ 
			9    & -        & -              & 21   & -        & -              \\ 
			10   & -       & -             & 22   & -       & -             \\ 
			11   & 200       & 3             & 23   & -        & -              \\ 
			12   & 500        &  1.2            & 24   & 200        & 3              \\ \hline
		\end{tabular}
	\end{table}

	The total time spent in the system is equal to 11,301 passenger-hours including 8,673 passenger-hours of travel time on various links, 1,881 passenger-hours of wait time spent on the transit network, and 747 passenger-hours of wait time spent on the road network. We found that more passengers take transit than MoD service. The share of passengers using the road, transit, and multimodal service are 23 \%, 61 \%, and 16 \% respectively. The passenger flow on various links and wait time on various nodes of the road and transit networks (resp.) are visualized in Figure \ref{fig:fn}(a) and (b) respectively. We observe that most of the passenger trips in the central zones are made using transit network, whereas the trips on the outskirts of the network are made using both MoD and multimodal service. The figures further show that the congestion in the central zones is significantly improved with the resulting network design.

\definecolor{zero}{rgb}{1,1,.5}
\begin{figure}[h!]
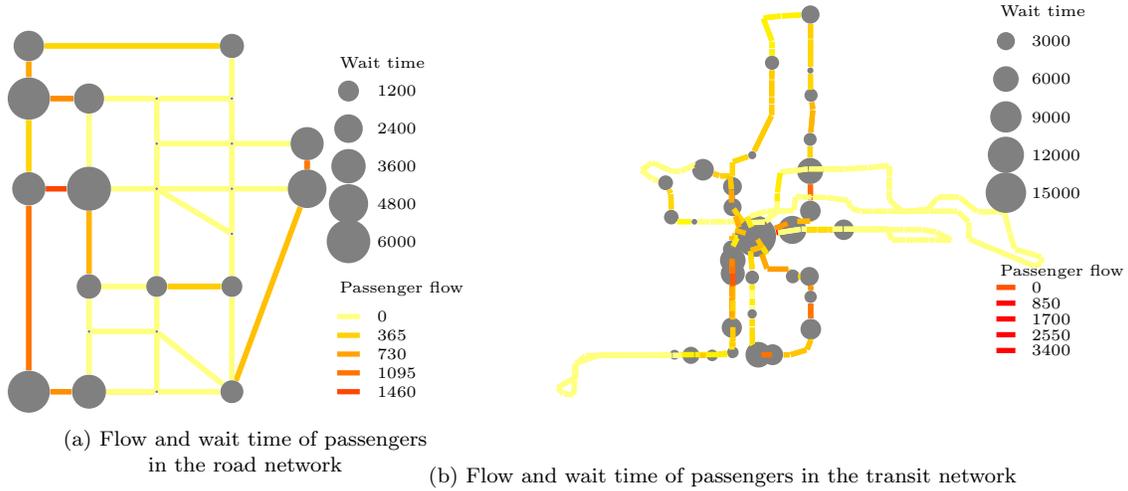

	\begin{subfigure}[t]{0.49\textwidth}            
		
		\centering			
		\tikzset{
			Node/.style = {
				fill = gray,
				draw = none,
				circle,
				minimum width = \fpeval{1+sqrt(#1)/5},
				inner sep=0pt
			},
			Link/.style = {
				line width = 0.7mm,
				draw = {red!\fpeval{#1/20}!yellow!\fpeval{100*(#1>1)}!zero},
			},
			legend/.style = {
				anchor = west,
			},
		}
		

		\label{fig:roadFlow}		
		\caption{Flow and wait time of passengers\\ in the road network}
		
	\end{subfigure}        
	\hspace{-2cm}
	\begin{subfigure}[t]{0.49\textwidth}	
		\centering
		\tikzset{
			Node/.style = {
				fill = gray,
				draw = none,
				circle,
				minimum size = \fpeval{sqrt(#1)/8},
				inner sep=0pt
			},
			Link/.style = {
				line width = 0.7mm,
				draw = {red!\fpeval{#1/250}!yellow!\fpeval{1000*(#1>1)}!zero},
			},
			legend/.style = {
				anchor = west,
			},
		}
		
		%
		
		\label{fig:transitFlow}		
		\hspace{1cm}
		\caption{Flow and wait time of passengers in the transit network}
	\end{subfigure}

	\caption{Flow and wait time (pass-min) of passengers in the network}
	\label{fig:fn}
\end{figure}

	\subsection{Computational performance}\label{sec:comp}
	In this section, we compare the computational performance of various models and implementation techniques. We consider the following approaches to compare: 
	
	\begin{enumerate}
		\item Solving model \eqref{eq:11} using Gurobi solver
		\item Solving model \eqref{eq:50} using Gurobi solver
		\item Solving model \eqref{eq:50} using classic Benders decomposition (Algorithm \ref{alg:cb})
		\item Solving model \eqref{eq:50} using Benders decomposition with clique/cover cuts (\cref{sec:cc})
		\item Solving model \eqref{eq:50} using Benders decomposition with multiple cuts via multiple solutions (\cref{sec:mult})
		\item Solving model \eqref{eq:50} using Benders decomposition with both clique/cover and multiple cuts via multiple solutions
		\item Solving model \eqref{eq:50} using Benders decomposition with disaggregated cuts (\cref{sec:dis})
		\item Solving model \eqref{eq:50} using Benders decomposition with disaggregated and clique/cover cuts 
		\item Solving model \eqref{eq:50} using Benders decomposition with disaggregated and multiple cuts via multiple solutions 
		\item Solving model \eqref{eq:50} using Benders decomposition with disaggregated, clique/cover, and multiple cuts via multiple solutions
	\end{enumerate}
	
	To solve the bilinear model \eqref{eq:11}, we set the Gurobi parameter \texttt{NonConvex = 2}. For Benders decomposition with multiple cuts via multiple solutions, we set the Gurobi parameters \texttt{PoolSolutions = 2},  \texttt{PoolGap = 0.01}, \texttt{PoolSearchMode = 2}. For all above tests, the maximum time limit was set to 3 hours. \\
	
	\begin{table}[]
		\caption{Computational performance}
		\label{tab:comp}
		\centering
		\begin{tabular}{llll}
			\hline
		\textbf{Method}                                  & \textbf{Iterations} & \textbf{Computational} & \textbf{Gap (\%)}    \\
			                           &  & \textbf{time (s)} &    \\\hline
  			Gurobi bilinear                         & -          & Timed out$^{*}$            & 13.3 \\
            Gurobi                                  & -          & Timed out$^{*}$            & 0.62\\
			Classic                                 & 734        & Timed out$^{*}$           & 0.16\\
			Classic + Clique/Cover                  & 510        & 7,440                   & 0    \\
			Classic + Multiple                      & 500        & 6,524                   & 0    \\
			Classic + Clique/Cover + Multiple       & 423        & 5,566                   & 0    \\
			Disaggregate                            & 31         & 381                    & 0    \\
			Disaggregate + Clique/Cover            & 30         & 347                    & 0    \\
			Disaggregate + Multiple                 & 28         & 535                    & 0    \\
			Disaggregate + Multiple + Clique/Cover & 27         & 498                    & 0   \\  \hline
		\end{tabular}
	\begin{tablenotes}\footnotesize
		\item[*] *Note: Maximum time limit = 3 hours
	\end{tablenotes}
	\end{table}

	The computational performance of every method is shown in Table \ref{tab:comp}. The iterations are counted as the number of times RMP is solved, the computational time is recorded in seconds, and Gap is defined as $(UB - LB)*100/UB$. The bilinear model \eqref{eq:11} is hard to solve, and Gurobi took 3 hours to reach the optimality gap of 13.3 \%. The rest of the results are discussed for the optimization model \eqref{eq:50}. Other than Gurobi and classic Benders decomposition, all the methods coverage to the optimal solution. Gurobi and classic Benders decomposition reached an optimality gap of 0.62\% and 0.16\% respectively. This means both methods reached very close to the optimal solution in 3 hours. The hybrid approach of classic Benders decomposition with both clique/cover cuts and multiple cuts via multiple solutions outperforms the classic Benders decomposition with clique/cover cuts or multiple cuts via multiple solutions only. The disaggregated Benders decomposition is computationally more efficient than any classic Benders approach with cut improvements. The disaggregated cuts with other cuts show further improvement in the solution time and the number of iterations to converge to the optimal solution. The Benders decomposition using disaggregated, clique/cover, and multiple cuts via multiple solutions outperforms other methods in terms of the number of iterations to converge to an optimal solution, whereas Benders decomposition with disaggregated and clique/cover cuts outperform other methods in terms of computational time. This may be because the multiple cuts are generated by solving several subproblems, which takes more computational time, but the generated cuts may not be as effective. Overall, the experiments performed for this section show that the computational methods presented in this study are quite efficient in solving the current problem exactly.

	\subsection{Sensitivity analysis on parameters}\label{sec:sens}
	The availability of buses and vehicles can result in different network design results. Hence, we choose to perform a sensitivity analysis on the available bus fleet $\bar{B}$ and vehicle fleet $\bar{F}$. We solve the model \eqref{eq:50} with varying bus fleet size of 25, 50, 75, 100, and 150 and varying vehicle fleet size of 500, 1000, 3000, 5000, 8000. Figure \ref{fig:sens} and \ref{fig:sens1}  show the sensitivity analysis results based on contour plots. The x-axis shows the varying vehicle fleet sizes, and contours represent varying bus fleet sizes. Figure \ref{fig:sens}(a), (b), (c), and (d) show the in-vehicle cost, average road wait time, average transit wait time, and total expected travel cost in passenger-hours respectively. We can observe that the in-vehicle cost decreases with the increase in the number of available vehicles. The effect of increasing the number of buses is more than the increase in the number of vehicles. Moreover, the in-vehicle cost is not affected by increasing the number of vehicles to more than 5,000. The average road wait time decreases with the increase in the number of available vehicles. It also decreases with the increase in the number of available buses due to mode shift. The passenger-hours spent as the wait time in the transit network increases with the increase in the number of available vehicles as well as buses. This is because more passengers take the transit mode as more buses are made available. The overall expected travel cost also reduces with the increase in the bus and vehicle fleet. However, the effect of an increase in the vehicle fleet size of more than 5,000 is negligible. Figure \ref{fig:sens1} shows the mode share as a function of the available bus and vehicle fleet size. As expected, the transit and MoD share increase with the increase in the available bus and vehicle fleet size respectively. The share of multimodal service increases with the number of vehicles and buses up to 5,000 and 75 respectively but declines after that. The decline in multimodal share is because of the reduced wait time for both services, which drives passengers to use single mode rather than multiple modes.

	\pgfplotsset{compat = 1.3}
	\begin{figure}[H]
		\begin{subfigure}[t]{0.43\textwidth}
			\begin{tikzpicture}[scale=0.8]
			\begin{axis}[xtick={500,1000,3000,5000,8000},x tick label style = {font = \small, align = center, rotate = 60, anchor = north east}, xlabel=$\bar{F}$,  ylabel = Total in-vehicle cost]
			\addplot[contour prepared, contour/label distance=4cm]
			table {Data/ivt.txt};
			\end{axis}
			\end{tikzpicture}
			\label{fig:sensa}			
			\caption{In-vehicle cost (pass-hrs)}
		\end{subfigure}        
		\hfill
		\begin{subfigure}[t]{0.43\textwidth}
			\begin{tikzpicture}[scale=0.8]
			\begin{axis}[xtick={500,1000,3000,5000,8000}, ,x tick label style = {font = \small, align = center, rotate = 60, anchor = north east}, xlabel=$\bar{F}$,  ylabel = Road wait time]
			\addplot[contour prepared, contour/label distance=5cm]
			table {Data/roadWait.txt};
			\end{axis}
			\end{tikzpicture}
			\label{fig:sensb}			
			\caption{Average road wait time (pass-hrs)}
		\end{subfigure}  
		\hfill
		\begin{subfigure}[t]{0.43\textwidth}
			\begin{tikzpicture}[scale=0.8]
			\begin{axis}[xtick={500,1000,3000,5000,8000}, ,x tick label style = {font = \small, align = center, rotate = 60, anchor = north east}, xlabel=$\bar{F}$,  ylabel = Transit wait time]
			\addplot[contour prepared, contour/label distance=6cm]
			table {Data/transitWait.txt};
			\end{axis}
			\end{tikzpicture}
			\label{fig:senc}		
			\caption{Average transit wait time (pass-hrs)}
		\end{subfigure}   
		\hfill
		\begin{subfigure}[t]{0.43\textwidth}
			\begin{tikzpicture}[scale=0.8]
			\begin{axis}[xtick={500,1000,3000,5000,8000}, ,x tick label style = {font = \small, align = center, rotate = 60, anchor = north east}, xlabel=$\bar{F}$,  ylabel = Total expected travel cost]
			\addplot[contour prepared, contour/label distance=6cm]
			table {Data/totalCost.txt};
			\end{axis}
			\end{tikzpicture}
			\label{fig:sensd}			
			\caption{Total expected travel cost (pass-hrs)}
		\end{subfigure}  
		\caption{Sensitivity of parameters $\bar{F}$ and $\bar{B}$ on different costs (contour represents varying bus fleet sizes)}
		\label{fig:sens}
	\end{figure}
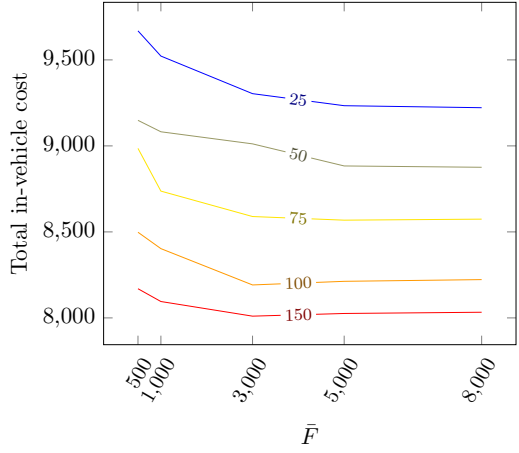
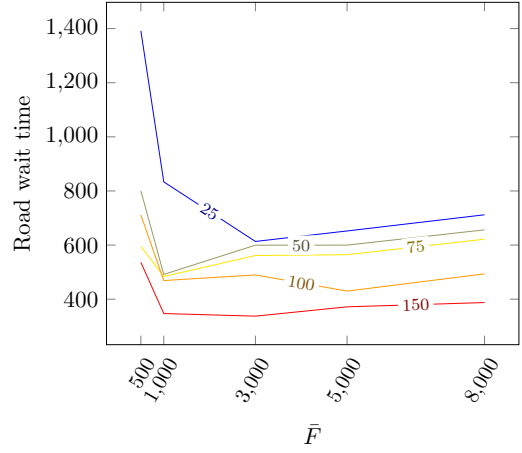
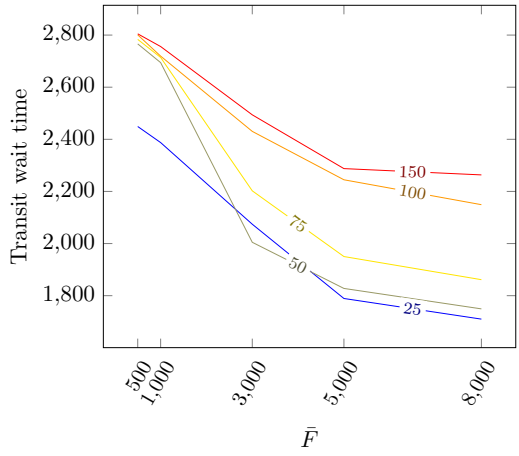
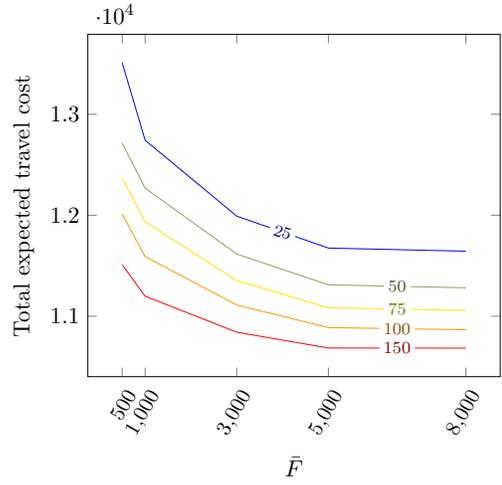

	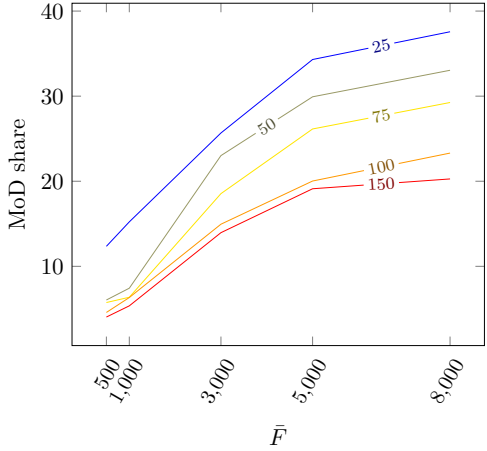
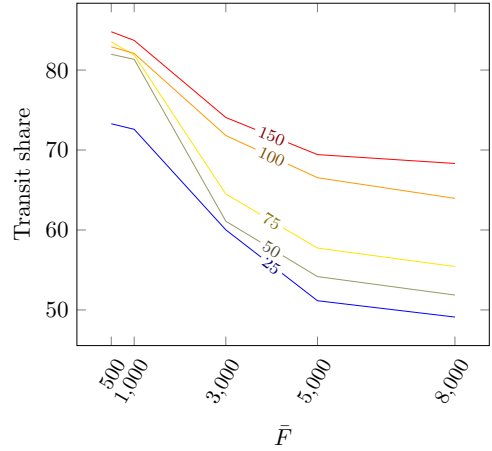
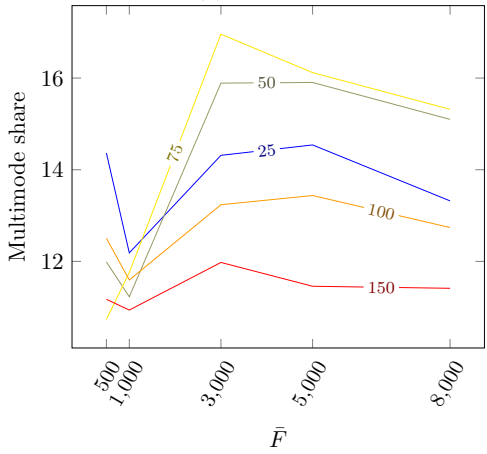
\begin{figure}[H]
		\begin{subfigure}[t]{0.43\textwidth}
			\begin{tikzpicture}[scale=0.8]
			\begin{axis}[xtick={500,1000,3000,5000,8000}, ,x tick label style = {font = \small, align = center, rotate = 60, anchor = north east}, xlabel=$\bar{F}$,  ylabel = MoD share]
			\addplot[contour prepared, contour/label distance=6cm]
			table {Data/autoShare.txt};
			\end{axis}
			\end{tikzpicture}
			\label{fig:sensa1}			
			\caption{MoD share}
		\end{subfigure}        
		\hfill
		\begin{subfigure}[t]{0.43\textwidth}
			\begin{tikzpicture}[scale=0.8]
			\begin{axis}[xtick={500,1000,3000,5000,8000}, ,x tick label style = {font = \small, align = center, rotate = 60, anchor = north east}, xlabel=$\bar{F}$,  ylabel = Transit share]

			\addplot[contour prepared, contour/label distance=5cm]
			table {Data/transitShare.txt};
			\end{axis}
			\end{tikzpicture}
			\label{fig:sensb1}			
			\caption{Transit share}
		\end{subfigure}  
		\hfill
		\begin{subfigure}[t]{0.43\textwidth}
			\begin{tikzpicture}[scale=0.8]
			\begin{axis}[xtick={500,1000,3000,5000,8000}, ,x tick label style = {font = \small, align = center, rotate = 60, anchor = north east}, xlabel=$\bar{F}$,  ylabel = Multimode share]
			\addplot[contour prepared, contour/label distance=6cm]
			table {Data/multimodeShare.txt};
			\end{axis}
			\end{tikzpicture}
			\label{fig:senc1}		
			\caption{Multimode share}
		\end{subfigure}   
		\caption{Sensitivity of parameters $\bar{F}$ and $\bar{B}$ on mode share (contour represents varying bus fleet sizes)}
		\label{fig:sens1}
	\end{figure}

	\subsection{Comparison of optimized base transit system with proposed integrated system}\label{sec:comp2}
	In this section, we present a comparison of the operation of the ``optimized base transit system" corresponding to the existing transit system with optimized frequencies versus the design of the integrated system evaluated in \cref{sec:netd}. For the optimized base case, we solve the optimization program \eqref{eq:70} for the instance described in \cref{sec:ins}. The results of optimized frequencies of various routes are given in Table \ref{tab:routeLoc2}. The network provides an average wait time of 18 minutes to the passengers. 
	
		\begin{subequations}\label{eq:70}
		\begin{align}
		& \underset{v, W, y}{\text{minimize}}
		& & \sum_{k \in D} \left(\sum_{a \in A} c_{a} v_{ak} + \sum_{i \in N^w_T} W_{ik}\right) \\
		& \text{subject to}
		& & \sum_{l \in L} \sum_{f \in \Theta} \mathcal{B}(l, f) \times y_{lf}  \le \bar{B}  \\
		& & & \sum_{a \in FS(i)}  v_{ak} = \sum_{a \in BS(i)}  v_{ak} +  g_{ik}, \forall i \in N, \forall k \in D  \\
		& & & v_{ak} \le  \left(\sum_{f\in \Theta} f y_{l(a)f}\right) W_{ik}, \forall a \in FS(i): a \in A_T, \forall i \in N^w, \forall k \in D\\
		& & &  v_{ak} \ge 0, \forall a \in A, \forall k \in D \\
		& & & W_{ik} \text{ free }, \forall i \in N^w_T, \forall k \in D \\
		& & &  y_{lf} \in \mathfrak{B}, \forall f \in \Theta,  \forall l \in L
		\end{align}
	\end{subequations}

	\begin{table}[h!]
		\caption{Routes with their optimal frequency (optimized base case)}
		\label{tab:routeLoc2}
		\centering
		\begin{tabular}{cccc}
			\hline
			\textbf{Route} & \textbf{Optimal} & \textbf{Average wait}\\
			&  \textbf{frequency} (buses/hr) & \textbf{time} (min) \\ \hline
			1     &   4         & 15                  \\ 
			2     &   2        & 30                 \\ 
			3     &   6            & 10                \\ 
			4     &   12              & 5            \\ 
			5     &    3         & 20                \\ 
			6     &  2         &   30             \\ 
			7     &  3            & 20               \\ 
			8     &   3          & 20               \\ 
			9     &  3          &   20             \\ 
			10    &   2         &  30              \\ 
			11    &   12          & 5               \\ 
			12    &    6          & 10              \\ \hline
		\end{tabular}
	\end{table}

	The results comparing the performance of the optimized base transit system and integrated system are provided in Table \ref{tab:comp1}. The base transit system has 12 two-way bus services operated by a bus fleet of 69 buses, whereas the new integrated system has 6 two-way bus services operated by 69 buses. Along with 69 buses, the new integrated system deploys 3,000 vehicles to serve the demand. The deployment of these extra vehicles can be costly to the transportation agencies. However, they provide several benefits. First, the optimized base transit system is not able to serve 13 \% of the demand due to the non-availability of transit service in 2 zones in the network. On the other hand, the first mile and last mile of these zones are covered by vehicles in the integrated system. Second, the average in-vehicle travel time of passengers using the integrated system is only 14.43 minutes in comparison to the 21.16 minutes for passengers using the base transit system. However, the average wait time of the integrated system users is increased slightly in comparison to the base transit system. This is due to the increased number of transfers to access MoD and transit service. 

	\begin{table}[h!]
		\caption{Comparison of optimized base transit system and integrated system}
		\label{tab:comp1}
		\centering
		\begin{tabular}{lcc}
			\hline
			& \textbf{Optimized base} & \textbf{Integrated} \\
			& \textbf{transit system} & \textbf{system} \\ \hline
			Number of active routes                 & 12                            & 6                     \\
			Number of buses used                    & 69                            & 69                    \\
			Number of vehicles used                      & 0                             & 3,000                  \\
			Satisfied demand (\%)                   & 87                            & 100                   \\
			Average in-vehicle time (min/passenger) & 21.16                         & 14.43                 \\
			Average wait time (min/passenger)       & 2.88                          & 4.37   \\ \hline
		\end{tabular}
	\end{table}

	\subsection{Managerial insights for implementing such service}\label{sec:policy}
	For implementing the proposed model in practice, we need to follow the following procedure. First, we divide the region into zones. Second, we collect the peak hour demand data in the region. Third, solve the proposed design model for varying fleet sizes. This step will be similar to the sensitivity analysis given in Section \ref{sec:sens}. This analysis will help us decide the optimal fleet size of buses and vehicles for our service. It will also provide us with the allocation of vehicles and buses for different zones and bus routes respectively. This allocation is designed for peak hours. We can reduce the vehicle operation in non-peak hours. For the bus service, scheduling needs to be performed to publish a schedule for the service.\\
	
	\fg{The numerical results show that the proposed integrated MoD and transit system provides several benefits. First, it relieves congestion in the city center by allocating the transit frequency from low-population density areas to congested areas. Second, it improves the level of service and access of transit passengers, which leads to an increase in the overall share of transit trips in the city. Third, improvement in transit passenger mobility further leads to various socio-environmental benefits.}

	\section{Conclusions and Future Research}  \label{sec:cf}
	Advances in mobility services have paved the way for the development of a new type of MoD service, which can help in serving the first mile and last mile of transit trips. Such a system requires rethinking the design of a transit system that allows for intermodal trips with MoD as the first or last leg of trips. We developed a mixed-integer non-linear program (MINLP) to design such system. The MINLP was relaxed to a mixed-integer linear program with the help of McCormick relaxations. To solve the resulting MILP model efficiently, we proposed the Benders decomposition method with several enhancements. These enhancements include the use of disaggregated cuts, clique/cover cuts, and multiple cuts via multiple solutions. The numerical results show that disaggregated cuts with clique/cover cuts and multiple cuts via multiple solutions are efficient techniques to solve the current problem. Furthermore, the experiment results on the Sioux Falls network show that the congestion in the city center is improved with such design as most of the passengers were found to take the transit in that region. The sensitivity analysis on bus fleet size and vehicle fleet size shows that the passenger hours spent in the system as in-vehicle time and wait time reduces with an increase in the number of available buses and vehicles. The share of multimodal service was observed to be highest for the vehicle fleet size and bus fleet size of 3,000 and 75 respectively. We also compared the proposed integrated system with the optimized base transit system. We found that the integrated system can be costly due to the deployment of vehicles, but it reduces the passenger in-vehicle time and serves more demand than the optimized base case. \\

	This research can be expanded in multiple directions. First, ridepooling was not allowed in the current study. Further research is needed to explore the ideas of including the matching of passengers for ridepooling, which will further reduce the size of the vehicle fleet required to provide the service. Second, better calibration of parameter $\mathcal{A}$ used in the wait time computation of MoD service is needed. The data from ridehailing services can be used for this purpose. \fg{Third, stochasticity from both supply-side and demand-side should be studied in the future as stochasticity can be good or bad for a transport system (\citet{Wang2021})}.

	\section*{Acknowledgments}
	This research is conducted at the \href{http://umntransit.weebly.com/}{University of Minnesota Transit Lab}, currently supported by the following, but not limited to, projects:
	\begin{itemize}	
		\item[-] National Science Foundation, award CMMI-1831140
		\item[-] Freight Mobility Research Institute (FMRI), Tier 1 Transportation Center, U.S. Department of Transportation: award RR-K78/FAU SP\#16-532 AM2 and AM3
		\item[-] Minnesota Department of Transportation, Contract No. 1003325 Work Order No. 44 and 111
		\item[-] University of Minnesota Office of Vice President for Research, COVID-19 Rapid Response Grants 
	\end{itemize}	
	
	\newpage

	\bibliographystyle{plainnat}
	\bibliography{library.bib}

	\begin{appendices}
		\section{Proofs of various propositions}\label{app:1}	
		\noindent \textbf{Proof of Proposition \ref{prop:1}}		

			(Although the proof can be found in \citet{Spiess1989} or \citet{gentile2005route}, we repeat it here because some of the details presented here will be used in proving the next proposition.) The probability of choosing line $i \in FS(n)$ is equal to the probability of wait time for line $i \in FS(n)$ to be less than or equal to wait time of other lines $j \ne i$, i.e., 
			\begin{equation}\label{eq:5}
				P_i = Prob (w_i \le \text{min}_{j \ne i} w_j) =  \int_{0}^\infty \mathfrak{g}_i(w) \Pi_{j \ne i}  Prob(w_j \ge w) dw = \int_{0}^{\infty} \gamma_i(w)dw
			\end{equation}

			where, $\gamma_i(w) = g_i(w) \Pi_{j \ne i}  Prob(w_j \ge w) = f_ie^{-f_iw} \Pi_{j \ne i} e^{-f_jw} = f_i e^{-(\sum_{j}f_j)w}$. The value of $\gamma_i(w)$ can be interpreted as the probability density function of the wait time at the stop $n$ conditional to boarding line $i$. 	Using \eqref{eq:5}, the probability of choosing line $i \in FS(n)$  can be evaluated as:
			
			\begin{equation}\label{eq:8}
				P_i =  \int_{0}^\infty f_i e^{-(\sum_{j}f_j)w} dw = \frac{f_i}{\mathfrak{F}}
			\end{equation}
			
			\noindent The expected wait time conditional to boarding line $i$ is:
			
			\begin{equation}
				EW_i = \int_{0}^{\infty} w\gamma_i(w)dw = \int_{0}^{\infty} w f_i e^{-(\sum_{j}f_j)w} dw = \frac{f_i}{\mathfrak{F}^2}
			\end{equation} 
			
			Summing over all the lines $FS(n)$ gives us the expected wait time at the stop, i.e., 
			
			\begin{equation}
				EW_n = \sum_{i \in FS(n)} \int_{0}^{\infty} w\gamma_i(w)dw =  \int_{0}^{\infty} w \sum_{i}\gamma_i(w)dw 
			\end{equation} 
			
			where, $\sum_{i}\gamma_i(w)$ is the probability density function of the wait time at stop $n$. 
			
			\begin{equation}
				\sum_{i \in FS(n)}\gamma_i(w) = \sum_{i \in FS(n)} f_ie^{-(\sum_{j}f_j)w} = \mathfrak{F}e^{-\mathfrak{F}w}, w \ge 0
			\end{equation}
			
			Therefore, the expected wait time at stop $n$ if given by $EW_n = \frac{1}{\mathfrak{F}}$\\
			
	\noindent \textbf{Proof of Proposition \ref{prop:2}}		

		The probability of taking transit is given by:
		\begin{eqnarray}
			P_{transit} = \int_{0}^{\infty} (\sum_{i}\gamma_i(w)) Prob(w\le w_{MoD}) dw\\
			P_{transit} = \int_{0}^{\infty}  \mathfrak{F}e^{-\mathfrak{F}w} \times e^{-(f_{MoD})w} dw =	\frac{\mathfrak{F}}{\mathbb{F}}\\
		\end{eqnarray} 
		Similarly, the probability of taking MoD is given by $P_{MoD} = \frac{\mathfrak{F}}{\mathbb{F}}$.	
		The expected wait time of the passenger departing from an access node $n$ is given by:
		\begin{equation}
			EW_n = \int_{0}^{\infty} w \left(\mathfrak{F} e^{- (f_{MoD} + \mathfrak{F})} + f_{MoD} e^{- (f_{MoD} + \mathfrak{F})} \right) dw = \frac{1}{\mathbb{F}}
		\end{equation}\\

	\noindent \textbf{Proof of Proposition \ref{prop:4}}			
	To prove this, we need to show that (a) $proj_{v, W} \mathcal{X}^{SP} \subseteq \mathcal{X}^{MA}$ and (b) $\mathcal{X}^{MA} \subseteq proj_{v, W} \mathcal{X}^{SP}$. Let us first start by proving (b). Let $(v, W) \in \mathcal{X}^{MA}$. For all $a \in FS(i) : a \in A_T, \forall i \in N^w, \forall k \in D$, we  have $v_{ak} \le f_a W_{ik}$. Let $y_{l(a)f} = 1$ if the frequency of the line associated to arc $a$ is $f \in \Theta$ and 0, otherwise. Let $ t_{faik} =  \hat{y}_{l(a)f} W_{ik}$, then $f_a W_{ik} = \sum_{f\in \Theta} f \hat{y}_{l(a)f} W_{ik} = \sum_{f\in \Theta} f t_{faik}$, which is same as \eqref{eq:33}. Also, $ t_{faik} =  \hat{y}_{l(a)f} W_{ik}$ can be expressed as \eqref{eq:40}- \eqref{eq:34} and \eqref{eq:40a}. Using a similar argument, we can show that for all $a \in FS(i) : a \in A_R, \forall i \in N^w, \forall k \in D$, the inequality $v_{ak} \le f_a W_{ik}$ can be expressed as \eqref{eq:36}-\eqref{eq:37} and \eqref{eq:30}. This shows that  $\mathcal{X}^{MA} \subseteq proj_{v, W} \mathcal{X}^{SP}$. To prove part (a), let $(v, W, \omega, t) \in \mathcal{X}^{SP}$, then using Fourier-Motzkin elimination, we have $(v, W) \in \mathcal{X}^{MA}$ (\citet[Chapter~3]{conforti2014integer}).\\
	
	\noindent \textbf{Proof of Proposition \ref{prop:5}}			
	The set $\mathcal{X}^{SP}$ can be empty in two cases i.e., when there is no flow balance ($\sum_{k \in D} \sum_{i \in N} g_{ik} \ne 0$) or there does not exist a directed path from any node $i \in N$ to any destination $k$. However, it is not possible to have any of these cases because from the definition of $g_{ik}$, we have $\sum_{k \in D} \sum_{i \in N} g_{ik} = 0$ and since there is always at least 0.01 vehicle assigned to all the zones and the road network is connected, there always exists a path from any node $i \in N$ to any destination $k \in D$. Therefore, $\mathcal{X}^{SP} \ne \phi$. \\
	\noindent \textbf{Proof of Proposition \ref{prop:6}}
	 Due to limited fleet available, one can allocate $n \in \Omega$ vehicles in at most $\lfloor \frac{\bar{F}}{n} \rfloor$ zones. \\
	\noindent \textbf{Proof of Proposition \ref{prop:7}}	
	The proposition follows from the definition of minimal cover that we cannot provide the number of buses required in the minimal cover. 
		
	\end{appendices}
	
\end{document}